\documentclass{article}
\usepackage[utf8]{inputenc}
\usepackage{amsmath}
\usepackage{amssymb}
\usepackage{amsthm}
\usepackage{graphicx}
\usepackage{biblatex}
\usepackage{caption}
\usepackage{authblk}

\newtheorem{theorem}{Theorem}[section]
\newtheorem{lemma}[theorem]{Lemma}
\newtheorem{proposition}[theorem]{Proposition}

\newtheorem{definition}{Definition}[section]

\newtheorem{question}[theorem]{Question}

\begin{document}

\title{Rigidity of Multisections to the Universal Cubic Plane Curve}

\author[*]{Jinwen Yao}

\affil[*]{Qiuzhen College, Tsinghua University

Email address: yaojw24@mails.tsinghua.edu.cn}

\renewcommand*{\Affilfont}{\small}

\maketitle

\begin{abstract}
   We study the following question of Farb: For which positive integers $n$ is it possible to choose $n$ distinct points continuously on every smooth cubic plane curve? This paper is a continuation of the work of Banerjee-Chen. We prove that such choice is unique up to homotopy for $n=27$ (smallest open case), and does not exist for $n=36k+18$.
\end{abstract}

\section{Introduction}

It is well known that every nonsingular cubic curve has precisely 9 inflection points (first attributed to Maclaurin; see Introduction of \cite{Hesse_pencil} for a brief history) and 27 sextatic points (first studied by Cayley in \cite{Cayley}). Motivated by these classical constructions from algebraic geometry, Benson Farb asked the following question: 

\begin{question}[A special case of Question 3.1 in \cite{Farb}]
Is it possible to choose in a continuously varying manner an unordered n-tuple of distinct points on every smooth cubic curve in $\mathbb{CP}^2$?
\end{question}

We now make the question precise. Let $\mathcal{X}$ denote the space of all nonsingular cubic curves,  in other words,
$$ \mathcal{X} = \{F(x,y,z) | F \ \mathrm{is \  a  \ nonsingular \ homogeneous \ polynomial \ of \ degree} \ 3 \} / \mathbb{C}^{*}. $$

Each $F \in \mathcal{X}$ defines a nonsingular cubic curve
$C_F = \{p \in \mathbb{CP}^2 \ | \ F(p) = 0 \} \subset \mathbb{CP}^2$. Let \textit{the universal cubic plane curve} be a fiber bundle $\pi: E \rightarrow \mathcal{X}$ over $\mathcal{X}$ whose fiber at point $F$ is $C_F$. 

A \textit{virtual section} over $\mathcal{X}$ degree $n$ is a triple $(\tilde{\mathcal{X}}_n, p_n, s_n)$, where $p_n : \tilde{\mathcal{X}}_n \rightarrow \mathcal{X}$ is an $n$ -sheeted cover and $s_n : \tilde{\mathcal{X}}_n \rightarrow E$ is a map such that $p_n = \pi \circ s_n$. When $s_n$ is injective, for every point $F$ in $\mathcal{X}$, there are exactly $n$ distinct points in the fiber of $F$ that are also in the image of $s_n$, so $s_n$ chooses $n$ points on each cubic curve. Unless otherwise specified, which happened a few times in the paper, we always assume that $s_n$ is injective.

In general, we know from algebraic geometry that there exist virtual sections  $\mathcal{X}$ of degree $n$ for some integers $n$. By works of Maclaurin (see e.g. \cite{Hesse_pencil}), Cayley \cite{Cayley} and Gattazzo \cite{Gattazzo}, we know that there exist virtual sections of degree $n$ when 

$$
n = 9 \sum_{m \in I} J_2(m), \ \ \ \ \   J_2(m) = m^2 \prod_{p | m, p \  prime} (1 - p^{-2}), 
$$
and $I$ is a set of positive integers. These virtual sections are constructed as follows: Every nonsingular cubic curve $C_F$ defined by $F=0$ is a Riemann surface of genus 1, which is a torus. If we choose an inflection point as the identity, then $C_F$ is an abelian group homomorphic to $(\mathbb{R}/\mathbb{Z})^2$. We can select $J_2(m)$ points for each $C_F$ by selecting all points $v$ such that $3v$ is an $m$-torsion. The set of points selected in such a way is well-defined, since different choices of inflection points differ by a 3-torsion. Moreover, we can take the disjoint union of the virtual sections corresponding to $m$ for every $m \in I$. 
In particular, let $I=\{1\}$ and we get the virtual section containing all the inflection points, denoted by $\tilde{\mathcal{X}}_{\mathrm{flex}}$, which is of degree $9J_2(1)=9$; let $I=\{2\}$ and we get the virtual section containing all the sextatic points, denoted by $\tilde{\mathcal{X}}_{\mathrm{sext}}$, which is of degree $9J_2(2)=27$.

\vspace{2mm}

Below is a summary that lists previous work on this question: 
\begin{enumerate}
    \item In \cite{Chen_2018}, Chen proved that if there exists a virtual section of degree $n$, then $n$ is divisible by 9.
    \item Later in \cite{B_Chen}, Banerjee and Chen proved that virtual sections of degree 9 are homotopic to $\sigma_{\mathrm{flex}}$, and virtual sections of degree 18 do not exist. 
    \item However, also in \cite{B_Chen}, Banerjee and Chen gave a list of $n$ where there exists a virtual section of degree $n$, but not from the above constructions. The list contains 216, 225, 243, 252, 288, 297, etc. In particular, from Theorem 1.9 in \cite{B_Chen}, we know that there exists a multisection of degree $n$ for $n= 108k+m$, where $k \in \mathbb{Z}_{\geq 0}$, and $m \in \{0,9,27,36,72,81,99\}$.  
    \item In the Appendix of \cite{B_Chen}, they recorded the proof of McMullen that any \textbf{algebraic} virtual section must be from the known algebraic structures.

\end{enumerate}

This paper carries on the work of \cite{B_Chen}, using similar notations and methods. In this paper, we prove the following three theorems.
\begin{theorem}
Every degree 27 virtual section is homotopic to $\tilde{\mathcal{X}}_{\mathrm{sext}}$.
\end{theorem}

A virtual section is \textit{connected} if $\tilde{\mathcal{X}}_n$ is connected. 

\begin{theorem}
There does not exist any connected virtual section not homotopic to $\tilde{\mathcal{X}}_{\mathrm{flex}}$ or $\tilde{\mathcal{X}}_{\mathrm{sext}}$, unless its degree can be divided by 36.
\end{theorem}

The connected virtual sections that we already know are those from the above constructions, added that $I=\{n\}$ as a set of a single number. $J_2(3)=8$ and $12$ divides $J_2(n)$ for $n \geq 4$, therefore other than $\tilde{\mathcal{X}}_{\mathrm{flex}}$ or $\tilde{\mathcal{X}}_{\mathrm{sext}}$, these virtual sections have $\frac{n}{9}=0,8$ mod 12. By this theorem, the only cases we still do not know are $n/9 = 12k + 4$ for some integer $k$.

\begin{theorem}
There does not exist any virtual section of degree $36k+18$, for any $\ k \in \mathbb{Z}_{\geq 0}$.
\end{theorem}

After Theorem 1.4 in this paper and Theorem 1.9 in \cite{B_Chen}, the only $n$ for which Question 1.1 remains the open case are $n= 108k+ 45 \ \mathrm{or} \ 108k+ 63$, where $k \in \mathbb{Z}_{\geq 0}$. Theorem 1.4 in this paper gives a negative answer to Question 1.1 when $n= 36k+18$, and Theorem 1.9 in \cite{B_Chen} gives a positive answer when $n = 108k+m , \, m \in \{0, 9, 27, 36, 72, 81, 99\}$.

\vspace{2mm}

The proofs of Theorem 1.2, 1.3, and 1.4 follow a similar path, which we will sketch here: Consider two cases divided by the ``monodromy group" $\Gamma_n$ (see Definition 3.1) of a connected virtual section. If $-I \in \Gamma_n$ we prove that the virtual section can only be $\tilde{\mathcal{X}}_{\mathrm{flex}}$ or $\tilde{\mathcal{X}}_{\mathrm{sext}}$. If $-I \notin \Gamma_n$, we prove that 36 divides the degree of the virtual section. 

This paper is organized as follows: In Section 2, we clarify some details about virtual sections.. In Section 3, we study the properties of how connected virtual sections are embedded into $E$. In Section 4, we find an element in the Artin's braid group $Br_m$ for some particular loops inside covering spaces of $\mathcal{X}$, which leads to an equation in $Br_m$, and use the equation to show contradictions for connected virtual sections not homotopic to $\tilde{\mathcal{X}}_{\mathrm{flex}}$ or $\tilde{\mathcal{X}}_{\mathrm{sext}}$. This is the key step of the proof, which we think is the most original and creative part. In Section 5, we extend from connected cases to general cases. Finally, in Section 6, we prove the main theorems.

\subsection*{Acknowledgement}

I am very grateful to Weiyan Chen for introducing the question to me and also for his meticulous guidance in writing this paper. I thank Liwen Shi for correcting me a mistake in the paper and helping me better understand the whole problem. I thank Leandro Vendramin for helping me understand his table of finite index subgroups of $
\mathrm{PSL}_2(\mathbb{Z})$.

In this paper, no AI was used, except to check grammar mistakes and misspellings.

\subsection*{Notations}

Here is a list of frequently used notations.

\vspace{2mm}

Introduced in the introduction and Chapter 2, and used globally:

$\mathcal{X}$ is the space of all nonsingular cubic curves. \par
$\pi: E \rightarrow \mathcal{X}$ is the universal cubic plane curve, a fiber bundle over $\mathcal{X}$ whose fiber at point $F$ is $C_F$. \par
In general, a virtual section of degree n is denoted as $\tilde{\mathcal{X}}_n$, its covering map is denoted by $p_n$, and its section map is denoted by $s_n$.\par
The virtual section given by all inflection points is denoted by $\tilde{\mathcal{X}}_\mathrm{flex}$, its covering map to $\mathcal{X}$ is denoted by $p_\mathrm{flex}$, and its section map is denoted by $s_\mathrm{flex}$.\par
The virtual section given by all sextatic points is denoted by $\tilde{\mathcal{X}}_\mathrm{sext}$, its covering map to $\mathcal{X}$ is denoted by $p_\mathrm{sext}$, and its section map is denoted by $s_\mathrm{sext}$.

We usually use $F$ to denote a point in $\mathcal{X}$.  \par
Usually an inflection point in a cubic curve $F$ is denoted as $p$, and an element in $\tilde{\mathcal{X}}_\mathrm{flex}$ is denoted as $(F,p)$; a sextatic point is denoted as $q$, and an element in $\tilde{\mathcal{X}}_\mathrm{sext}$ is denoted as $(F,q)$; a point that is neither an inflection point nor a sextatic point is denoted as $x$, and an element in $\tilde{\mathcal{X}}_n$ is denoted as $(F,x)$.

\vspace{2mm}

Introduced and mainly used in Chapter 3:

$K$ is the subgroup of $\mathrm{SL}_3(\mathbb{C})$ generated by
\[\begin{bmatrix} 
0 & 0 & 1 \\
1 & 0 & 0 \\
0 & 1 & 0
\end{bmatrix},
\begin{bmatrix} 
1 & 0 & 0 \\
0 & \omega & 0 \\
0 & 0 & \omega^2
\end{bmatrix},\  \omega=e^{2\pi i /3},
\]
and $Z(K)$ is its center, containing $I, \omega I, \omega^2 I$. \par
$\Lambda=\pi_1(\mathcal{X})/Z(K)$.\par
$\tilde{\Gamma}_n=\pi_1(\tilde{\mathcal{X}}_n)/Z(K)$ for a general virtual section $\tilde{\mathcal{X}}_n$.\par
$\tilde{\Gamma}_\mathrm{flex}=\pi_1(\tilde{\mathcal{X}}_\mathrm{flex})/Z(K)$ and $\tilde{\Gamma}_\mathrm{sext}=\pi_1(\tilde{\mathcal{X}}_\mathrm{sext})/Z(K)$. \par
$\rho$ is the monodromy representation $\pi_1(\mathcal{X}) \rightarrow Aut \, H_1(C_F;\mathbb{Z}) \cong \mathrm{SL}_2(\mathbb{Z})$. Usually we also denote the restriction of $\rho$ on  $\tilde{\Gamma}_\mathrm{flex}, \tilde{\Gamma}_\mathrm{sext}$ or $\tilde{\Gamma}_n$  as $\rho$. \par
$\Gamma_n$ is the image of $\tilde{\Gamma}_n$ under $\rho$, which is a subset of  $\mathrm{SL}_2(\mathbb{Z})$. We call it the monodromy group.

\vspace{2mm}

Introduced and mainly used in Chapter 4:

$\xi_{\mathrm{flex}}$ is the vector bundle over $\tilde{\mathcal{X}}_\mathrm{flex}$, whose fiber at point $(F,p)$ is the tangent space $T_pC_F$. \par
$\xi_{\mathrm{sext}}$ is the vector bundle over $\tilde{\mathcal{X}}_\mathrm{sext}$, whose fiber at point $(F,q)$ is the tangent space $T_qC_F$. \par

Usually we use $\gamma$ to denote a loop inside $\tilde{\mathcal{X}}_\mathrm{flex}$ or $\tilde{\mathcal{X}}_\mathrm{sext}$, and 
$\mathcal{T}$ to denote a trivialization of $\gamma^*\xi_{\mathrm{flex}}$ or $\tilde{\mathcal{X}}_\mathrm{flex}$.

We usually use $t$ to denote a point in $S^1$. In addition, $t$ will be written in the form of numbers in $\mathbb{R}/\mathbb{Z}$, and the base point of the loop is always $\gamma(0)$. 
Throughout the paper, the concatenation of two loops $\alpha$ and $\beta$ is denoted as $\alpha \cdot \beta$.

\section{Preliminaries about virtual sections}

All the definitions and results in this section are borrowed from \cite{B_Chen}.

\begin{definition}[Virtual section]
A \textit{virtual section} over $\mathcal{X}$ of degree $n$ is a triple $(\tilde{\mathcal{X}}_n, p_n, s_n)$, where $p_n : \tilde{\mathcal{X}}_n \rightarrow \mathcal{X}$ is a (possibly unconnected) $n$-sheeted cover, and $s_n : \tilde{\mathcal{X}}_n \rightarrow E$ is a map that $p_n = \pi \circ s_n$, where $\pi$ is the projection map in the bundle $ E \rightarrow \mathcal{X}$.

Two virtual sections $(\tilde{\mathcal{X}}_n, p_n, s_n)$ and $(\tilde{\mathcal{X}}_n', p_n', s_n')$ are \textit{homotopic}, if there is an isomorphism $f: \tilde{\mathcal{X}}_n \rightarrow \tilde{\mathcal{X}}_n'$ of cover of $\mathcal{X}$, and there is a homotopy $s_t: \tilde{\mathcal{X}}_n \rightarrow E$, such that  $s_0=s_n' \circ f$ and $s_1 =s_n$. 

A virtual section is \textit{connected} if $\tilde{\mathcal{X}}_n$ is connected. 
\end{definition}

\vspace{2mm}

As we have explained in the introduction, when $s_n$ is injective, it chooses $n$ distinct points on each nonsingular cubic curve $F$. So, for $\tilde{\mathcal{X}}_n$ to be meaningful, we assume that $s_n$ is injective by default. However, when we consider the homotopy between two virtual injective sections, we allow $s_t$ to be non-injective for $t \in (0,1)$. This allows us to treat the homotopy class of virtual sections as the homotopy class of lifts of $p_n:  \tilde{\mathcal{X}}_n \rightarrow \mathcal{X}$ from $\tilde{X}$ to $E$, and use obstruction theory (see Theorem \ref{obstruction}).

One might be curious what may happen if we reset the definition of homotopy of virtual sections so that every $s_t$ must be injective; this would result in a different category. The category of injective virtual sections with this kind of homotopy is equivalent to the category of multisections, which are also introduced in \cite{B_Chen}. Two categories have recently been found to be actually different, that is, there exist two virtual sections that are homotopic when we allow $s_t$ to be non-injective, and not homotopic when we do not allow it.

\section{Topological relations between virtual sections and $\tilde{\mathcal{X}}_\mathrm{flex}$}

In this section, we aim to discuss the properties of $H^1(\tilde{\mathcal{X}}_n;\mathbb{Z}^2)$ for the virtual section
$p_n:\tilde{\mathcal{X}}_n \rightarrow {\mathcal{X}}$, where the $\mathbb{Z}^2$ coefficients are in the monodromy representation of the pullback of the universal cubic curve bundle $p_n^*E$. Then we will use the properties of $H^1(\tilde{\mathcal{X}}_n;\mathbb{Z}^2)$ to prove the following theorem: 

\begin{proposition} \label{homotopic_sectionmap}
Given a connected virtual section $( \tilde{\mathcal{X}}_n,p_n,s_n )$. Suppose $-I \in \Gamma_n$. We have the following:

\begin{enumerate} \setlength{\itemindent}{1em}
    \item [(a).]  If $[s_n] = 0 \in H^1(\tilde{\mathcal{X}}_n;\mathbb{Z}^2)$, then there is a covering map $f:\tilde{\mathcal{X}}_n \rightarrow \tilde{\mathcal{X}}_\mathrm{flex}$, such that $p_\mathrm{flex} \circ f = p$, and $s_n$ is homotopic to $s_\mathrm{flex} \circ f$. 
    \item [(b).]  If $[s_n] \neq 0 \in H^1(\tilde{\mathcal{X}}_n;\mathbb{Z}^2)$, then there is a covering map $f:\tilde{\mathcal{X}}_n \rightarrow \tilde{\mathcal{X}}_\mathrm{sext}$, such that $p_\mathrm{sext} \circ f = p$, and $s_n$ is homotopic to $s_\mathrm{sext} \circ f$.
\end{enumerate}

\end{proposition}

The symbol $\Gamma_n$ in the assumption is the monodromy group of $\tilde{\mathcal{X}}_n$, which will be defined in Definition 3.1. The symbol $[s_n]$ means the homotopy class that  $s_n$ represents, and the relationship between homotopy classes of section maps and $H^1(\tilde{\mathcal{X}}_n;\mathbb{Z}^2)$ will be explained at the beginning of Subsection 3.2.

\vspace{3mm}

In 3.1, we study the fundamental groups of $\mathcal{X}$ and $E$, and the monodromy representation $\rho: \mathcal{X} \rightarrow Aut \, H_1(C_F;\mathbb{Z}) \cong \mathrm{SL}_2(\mathbb{Z})$. In 3.2, we calculate  $H^1(\tilde{\mathcal{X}}_n;\mathbb{Z}^2)$. In 3.3, we prove Proposition \ref{homotopic_sectionmap}.

\subsection{Structure of fundamental groups of $\mathcal{X}$ and $E$}
 
To begin our discussion, we want to calculate the fundamental groups of $\mathcal{X}$ and $E$, which is preliminary to our calculation for $H^1(\tilde{\mathcal{X}}_n;\mathbb{Z}^2)$. 
The monodromy of the bundle $E \rightarrow \mathcal{X}$ gives a monodromy representation $\rho : \mathcal{X} \rightarrow Aut \, H_1(C_F;\mathbb{Z}) \cong \mathrm{SL}_2(\mathbb{Z})$. Dolgachev and Libgober proved in \cite{Dolgachev} that $\rho$ is surjective, and fits into the following split exact sequence:

\begin{equation} \label{SES}
    1 \rightarrow K \rightarrow \pi_1(\mathcal{X}) \xrightarrow{\rho} Aut \ H_1(C_F;\mathbb{Z}) \rightarrow 1 
\end{equation}

Here $K$ is the Heisenberg group of order 27, or in matrix form, a subgroup of $\mathrm{SL}_3(\mathbb{C})$ generated by 
\[\begin{bmatrix} 
0 & 0 & 1 \\
1 & 0 & 0 \\
0 & 1 & 0
\end{bmatrix} \mathrm{and}
\begin{bmatrix} 
1 & 0 & 0 \\
0 & \omega & 0 \\
0 & 0 & \omega^2
\end{bmatrix},\  \omega=e^{2\pi i /3}.
\]

In Lemma 3.6, \cite{B_Chen}, it is shown that both $\pi_1(\mathcal{X})$ and $\pi_1(E)$ contain $Z(K)$ as a normal subgroup; moreover, by Theorem 3.9 in the same paper, for any virtual section $\tilde{\mathcal{X}}_n$, its fundamental group must contain $Z(K)$ as a normal subgroup as well. Therefore, we may quotient out $Z(K)$ from $\pi_1(\mathcal{X})$ and $\pi_1(E)$. Denote $\pi_1(\mathcal{X})/Z(K)$ as $\tilde{\Gamma}$ and $\pi_1(E)/Z(K)$ as $\Lambda$. We will also denote $ \tilde{\Gamma}_\mathrm{flex} = \pi_1(\tilde{\mathcal{X}}_\mathrm{flex})/Z(K) $ and $ \tilde{\Gamma}_\mathrm{sext} = \pi_1(\tilde{\mathcal{X}}_\mathrm{sext})/Z(K). $

Proposition 3.8 in \cite{B_Chen} shows the structure of $\Lambda = \pi_1(E)/Z(K) $ and $\tilde{\Gamma} = \pi_1(\mathcal{X})/Z(K)$ in a commutative diagram. The following is a list of properties taken from Proposition 3.8 in  \cite{B_Chen} that will be used in this paper.

\vspace{2mm}

\begin{enumerate} \setlength{\itemindent}{1em}
    \item [(1).] $\tilde{\Gamma}_\mathrm{flex} \cong \mathrm{SL}_2(\mathbb{Z})$ and $\tilde{\Gamma} \cong (\mathbb{Z}/3\mathbb{Z})^2  \rtimes \tilde{\Gamma}_\mathrm{flex}$.
    \item [(2).] $\Lambda \cong \mathbb{Z}^2  \rtimes  \tilde{\Gamma}_\mathrm{flex}$, and the action of $\tilde{\Gamma}_\mathrm{flex}$ on $\mathbb{Z}^2$ is the standard action of $\mathrm{SL}_2(\mathbb{Z})$ on $\mathbb{Z}^2$.
    \item [(3).] The map $\pi_*: \Lambda \rightarrow \tilde{\Gamma} $ explicitly is
\[(v, g) \mapsto (\Bar{v},g),\]
where $\Bar{v}$ is the image of $v$ under the quotient map $\mathbb{Z}^2 \rightarrow (\mathbb{Z}/3\mathbb{Z})^2$. In the isomorphism $\tilde{\Gamma} = (\mathbb{Z}/3\mathbb{Z})^2  \rtimes  \tilde{\Gamma}_\mathrm{flex}$, the action of $\tilde{\Gamma}_\mathrm{flex}$ on $(\mathbb{Z}/3\mathbb{Z})^2$ is a composition of the standard action $\tilde{\Gamma}_\mathrm{flex}$ on $\mathbb{Z}^2$ and the natural change of coefficients from $\mathbb{Z}^2$ to $(\mathbb{Z}/3\mathbb{Z})^2$.
\end{enumerate}

We now define the monodromy group.

\begin{definition}
Given a connected virtual section  $\tilde{\mathcal{X}}_n$, the \textit{monodromy group} of $\tilde{\mathcal{X}}_n$ is $\Gamma_n = \rho(\tilde{\Gamma}_n)$, where $\rho$ is the monodromy representation, and $\tilde{\Gamma}_n = \pi_1( \tilde{\mathcal{X}}_n )/Z(K)$.
\end{definition}

By the exact sequence (\ref{SES}), we know that $Z(K)$ lies in the kernel of the monodromy action $\rho$, so the definition makes sense.
From Corollary 3.12 in \cite{B_Chen}, we know that $\Gamma_n$ is a subgroup of $\mathrm{SL}_2(\mathbb{Z})$ with finite index $n/9$. Similarly, we will use $\Gamma_\mathrm{flex}$ to denote $\rho(\tilde{\Gamma}_\mathrm{flex})$ and use  $\Gamma_\mathrm{sext}$ to denote $\rho(\tilde{\Gamma}_\mathrm{sext})$.

Using the properties, we will prove that if $\Gamma_n \subset \mathrm{SL}_2(\mathbb{Z})$ contains $-I$, then $\tilde{\mathcal{X}}_n$ is a cover of $\tilde{\mathcal{X}}_\mathrm{flex}$:

\begin{proposition} \label{covering_map}
Given a connected virtual section  $(\tilde{\mathcal{X}}_n, p_n, s_n)$. If $\Gamma_n$ is a subgroup of $\mathrm{SL}_2(\mathbb{Z})$ containing $-I$, then we have a covering map $f:\tilde{\mathcal{X}}_n \rightarrow \tilde{\mathcal{X}}_\mathrm{flex}$ such that $p_n = p_\mathrm{flex} \circ f$.
\end{proposition}

\begin{proof}
The connected covering spaces of $\mathcal{X}$ are classified by their fundamental groups. It suffices to prove that $\pi_1(\tilde{\mathcal{X}}_n)$ can be conjugated into $\pi_1(\tilde{\mathcal{X}}_\mathrm{flex}) \subset \pi_1(\mathcal{X}) $. Since all three groups contain $Z(K)$, we only need to prove that $\tilde{\Gamma}_n$ can be conjugated into $\tilde{\Gamma}_\mathrm{flex} \subset \tilde{\Gamma}$.

For the split exact sequence 
\[ 1 \rightarrow (\mathbb{Z}/3\mathbb{Z})^2 \rightarrow (\mathbb{Z}/3\mathbb{Z})^2 \rtimes \Gamma_n \rightarrow \Gamma_n \rightarrow 1 ,\]
the conjugacy classes of group-theoretic sections $s:\Gamma_n \rightarrow (\mathbb{Z}/3\mathbb{Z})^2 \rtimes \Gamma_n$ are in bijection with $H^1(\Gamma_n;(\mathbb{Z}/3\mathbb{Z})^2)$. Recall that a cocycle in $Z(\Gamma_n;(\mathbb{Z}/3\mathbb{Z})^2)$ is a map $\phi: \Gamma \rightarrow (\mathbb{Z}/3\mathbb{Z})^2$ that satisfies $\phi(gh)= \phi(g) + g \cdot \phi(h) $ for any $g,h \in \Gamma_n$, and a coboundary in $B(\Gamma_n;(\mathbb{Z}/3\mathbb{Z})^2)$ is a map $\phi$ defined as $\phi(g)=g \cdot v-v$ for some $v \in (\mathbb{Z}/3\mathbb{Z})^2$. Given a cocycle $\phi$, we have
\[ \phi((-I) \cdot g) = \phi(-I) - \phi(g),  \ \ \
   \phi(g \cdot  (-I)) = \phi(g) + g \cdot  \phi(-I). \]
   
Subtracting two equations we get $2\phi(g) = \phi(-I) - g \cdot \phi(-I). $ Because $2=-1$ in $\mathbb{Z}/3\mathbb{Z}$, we have $\phi(g) = g \cdot \phi(-I) - \phi(-I). $ So every cocycle is a coboundary, hence $H^1(\Gamma_n;(\mathbb{Z}/3\mathbb{Z})^2)=0$, and any section $\Gamma_n \rightarrow (\mathbb{Z}/3\mathbb{Z})^2 \rtimes \Gamma_n$ is conjugate to the standard section $\Gamma_n \rightarrow (\mathbb{Z}/3\mathbb{Z})^2 \rtimes \Gamma_n, \ g \mapsto (0,g)$. This proves the proposition.

\end{proof}

\subsection{$H^1(\Gamma_n; \mathbb{Z}^2)$ and section maps}

In this subsection, we study the properties of $H^1(\tilde{\mathcal{X}}_n; \mathbb{Z}^2)$. This cohomology group is useful because of the following result in obstruction theory:

\begin{theorem}[See \textit{e.g.} Corollary 6.16 in \cite{Whitehead}] \label{obstruction}
Let $\pi:E \rightarrow B$ be a fiber bundle over a connected base, whose fiber is a $K(G,1)$ for an abelian group $G$. Let $\phi:K \rightarrow B$ be a continuous map. Suppose that $\phi$ can be lifted to a map $s:K \rightarrow E$, that is, there exists a map $s : K \rightarrow E$ such that $\pi \circ s=\phi$. Then the homotopy classes of liftings $s_n$ are in bijection with $H^1(E;G)$. 

\end{theorem}

The method of finding virtual sections can be divided into two steps: First, find all possible sections up to homotopy; second, check whether these sections can be homotoped to an injective one. Theorem \ref{obstruction} gives us a useful tool for the first step: Take $\pi$ to be the universal cubic curve $E \rightarrow \mathcal{X}$, and $\phi = p_n$. By Proposition \ref{covering_map}, we have a covering map $f:\tilde{\mathcal{X}}_n \rightarrow \tilde{\mathcal{X}}_\mathrm{flex}$. Define $s_n= s_\mathrm{flex} \circ f$, then $s_n$ is a lift of $p_n$. Because $C_F \cong T^2$ is $K(\mathbb{Z}^2,1)$, every element in $H^1(\tilde{\mathcal{X}}_n;\mathbb{Z}^2)$ represents a homotopy class of the section map $s_n$, although possibly not injective. For a section map $s_n$, we use $[s_n]$ to represent the element in $H^1(\Gamma_n;\mathbb{Z}^2)$ that corresponds to the homotopy class containing $s_n$.

By Lemma 4.2 of \cite{B_Chen}, $H^1(\tilde{\mathcal{X}}_n; \mathbb{Z}^2) \cong H^1(\Gamma_n; \mathbb{Z}^2).$ For $\Gamma_n \subset \mathrm{SL}_2(\mathbb{Z})$ that contains $-I$, the cohomology group is relatively simple. In fact, we will prove the following proposition:

\begin{proposition}
Given a virtual section $\tilde{\mathcal{X}}_n$, if its monodromy group $\Gamma_n$ is a subgroup containing $-I$, then $H^1(\Gamma_n;\mathbb{Z}^2)$ is $1, \mathbb{Z}/2\mathbb{Z},$ or $(\mathbb{Z}/2\mathbb{Z})^2$.
\end{proposition}

\begin{proof}
Recall the definition of cocycle and coboundary: A map $\phi: \Gamma_n \rightarrow \mathbb{Z}^2$ is a cocycle if it satisfies $\phi(gh)= \phi(g) + g \cdot \phi(h)$, and is a coboundary if $\phi(g)=g \cdot v-v$ for some $v \in (\mathbb{Z}/3\mathbb{Z})^2$.

For any cocycle $\phi$, we have the two following equations:
\[ \phi((-I) \cdot g) = \phi(-I) - \phi(g),  \ \ \
   \phi(g \cdot  (-I)) = \phi(g) + g \cdot  \phi(-I). \]
   
Subtracting two equations we get $2\phi(g) = \phi(-I) - g \cdot \phi(-I).$ Therefore, the homomorphism $Z(\Gamma_n;\mathbb{Z}^2) \rightarrow \mathbb{Z}^2, \ \phi \mapsto \phi(-I)$ is injective. Moreover, under this embedding, we claim that $(2\mathbb{Z})^2 \subset B(\Gamma_n;\mathbb{Z}^2)$. This is because for any $2v \in (2\mathbb{Z})^2$,
we have $\phi(g) = v - g \cdot v$, and $\phi$ is a coboundary. Therefore, we have
\[ 
H^1(\Gamma_n;\mathbb{Z}^2) = \frac{Z(\Gamma_n;\mathbb{Z}^2)}{B(\Gamma_n;\mathbb{Z}^2)} \subset \frac{\mathbb{Z}^2}{(2\mathbb{Z})^2} = (\mathbb{Z}/2\mathbb{Z})^2,
\]
and any subgroup of $(\mathbb{Z}/2\mathbb{Z})^2$ can only be $1, \mathbb{Z}/2\mathbb{Z},$ or $(\mathbb{Z}/2\mathbb{Z})^2$.
\end{proof}

By this proposition, we know that if $-I \in \Gamma_n$, any element in $H^1(\Gamma_n;\mathbb{Z}^2)$ is either trivial or a 2-torsion.

Now we study the properties of $ \tilde{\mathcal{X}}_\mathrm{flex}$ and $ \tilde{\mathcal{X}}_\mathrm{sext}$. The following propositions focusing on $ \tilde{\mathcal{X}}_\mathrm{flex}$ and $ \tilde{\mathcal{X}}_\mathrm{sext}$ will be useful to prove Proposition \ref{homotopic_sectionmap}.
\vspace{2mm}

\begin{proposition} [Proposition 5.2, \cite{B_Chen}] \label{flex}
$\ $

\begin{enumerate} \setlength{\itemindent}{2em}
    \item  [(a).]   $ \tilde{\mathcal{X}}_\mathrm{flex}$ is connected. 
    \item  [(b).]    $\Gamma_\mathrm{flex} \cong \mathrm{SL}_2(\mathbb{Z})$.
    \item  [(c).]   $H^1(\Gamma_\mathrm{flex};\mathbb{Z}^2) = 0$. 
    \item  [(d).]  $[s_\mathrm{flex}]=0.$  
\end{enumerate}

\end{proposition}

Notice that Proposition \ref{flex} is not the original statement of Proposition 5.2 in \cite{B_Chen}, but rather some claims that were stated and proved during the proof of Proposition 5.2 in \cite{B_Chen}.

\vspace{2mm}

\begin{proposition} \label{sext}
$\ $

\begin{enumerate} \setlength{\itemindent}{2em}
    \item  [(a).]   $\tilde{\mathcal{X}}_\mathrm{sext}$ is connected.
    
    \item  [(b).]  By properly selecting a basis of $H_1(C_F; \mathbb{Z})$,
     \[ \Gamma_\mathrm{sext} \cong \{ \ g \in \mathrm{SL}_2(\mathbb{Z}) \ | \ g = \begin{bmatrix} 1 & 0 \\ 
     0 & 1
     \end{bmatrix} \ \mathrm{or} \  \begin{bmatrix}
     1 & 1 \\
     0 & 1
     \end{bmatrix}  \  \mathrm{mod}  \  2 \  \}. \]
     
    \item  [(c).]  $H^1(\Gamma_\mathrm{sext};\mathbb{Z}^2) = \mathbb{Z}/2\mathbb{Z}$, the nontrivial element is represented by the cocycle \[ 
     \phi(g) = \frac{1}{2} (g \cdot \begin{bmatrix}
     1 \\ 0
     \end{bmatrix} - \begin{bmatrix}
     1 \\ 0
     \end{bmatrix}).
     \]  

    \item  [(d).]  $[s_\mathrm{sext}]$ is the nontrivial element in $H^1(\Gamma_\mathrm{sext};\mathbb{Z}^2)$.

\end{enumerate}

\end{proposition}

(a) and (b) of Proposition \ref{sext} are a special case of \cite{B_Chen}, Lemma 6.2 (taking $m=2$). We will prove (c) and (d) below.

\begin{proof}[Proof of Proposition \ref{sext}, (c)]

In the proof of Proposition 3.4, we have shown that for any cocycle $\phi: \Gamma_{\mathrm{sext}} \rightarrow \mathbb{Z}^2$, we have $2 \phi(g) = g \cdot v - v$ for some $v \in \mathbb{Z}^2$. Given a vector $v$, the cocycle $\phi(g) = \frac{1}{2} ( g \cdot v - v )$ exists if and only if  $( g \cdot v - v ) \in  (2\mathbb{Z})^2$ for any $g \in \Gamma_{\mathrm{sext}}$. 

\vspace{1mm}

If $v \in (2\mathbb{Z})^2$ then trivially $(g \cdot v - v ) \in  (2\mathbb{Z})^2$. One can verify that $(g \cdot v - v ) \in  (2\mathbb{Z})^2$ for $v = \begin{bmatrix}
     1 \\ 0
     \end{bmatrix}$, but not for $v = \begin{bmatrix}
     0 \\ 1
     \end{bmatrix}$ and $v = \begin{bmatrix}
     1 \\ 1
     \end{bmatrix}$. So all the cocycles are as form $\phi(g) = \frac{1}{2} ( g \cdot v - v )$, where $v =  \begin{bmatrix}
     0 \\ 0
     \end{bmatrix} $ or $ \begin{bmatrix}
     1 \\ 0
     \end{bmatrix}$ mod 2.

\vspace{2mm}

Furthermore, the cocycle $\phi(g) = \frac{1}{2} ( g \cdot v - v )$ is a coboundary if and only if  $v =  \begin{bmatrix}
     0 \\ 0
     \end{bmatrix} $ mod 2. Therefore $H^1(\Gamma_\mathrm{sext};\mathbb{Z}^2) = \mathbb{Z}/2\mathbb{Z}$, the nontrivial element is represented by the cocycle $ 
     \phi(g) = \frac{1}{2} (g \cdot \begin{bmatrix}
     1 \\ 0
     \end{bmatrix} - \begin{bmatrix}
     1 \\ 0
     \end{bmatrix}).$

\end{proof}

To prove (d) in Proposition \ref{sext}, we need the following lemma.

\begin{lemma} \label{mapping_torus}
Suppose $\gamma$ is a loop in $\tilde{\mathcal{X}}_{\mathrm{flex}}$ with base point $(F,p)$, whose monodromy action on $H_1(C_F; \mathbb{Z})$ is $-I$. Consider the pullback bundle of $E \rightarrow \mathcal{X}$ by $S^1 \xrightarrow{\gamma} \tilde{\mathcal{X}}_{\mathrm{flex}} \xrightarrow{p_{\mathrm{flex}}} \mathcal{X}$, where we denote it by $\gamma^*E$ (instead of $\gamma^* p_{\mathrm{flex}}^*  E $) for convenience. Then:

\begin{enumerate} \setlength{\itemindent}{1em}
    \item [(a).]  $\gamma^*E$ is homeomorphic to the mapping torus

$$ T_f = ([0,1] \times T^2) /\{ (0, v) \sim (1, -v) \, | \, v \in T^2 \}.$$

    \item [(b).]  There are exactly 3 sections of the bundle $\gamma^*E$ in which every point is a sextatic point.
    
    \item [(c).]  There exists a homeomorphism $\Psi: \gamma^*E \rightarrow T_f $ that maps each section given in (b) to a loop 
$$\bar{\gamma}_v : S^1 \rightarrow T_f, \,\, t \mapsto (t, v), $$
where $v \in \{ (\frac{1}{2},0), (0, \frac{1}{2}), (\frac{1}{2}, \frac{1}{2}) \}$.
\end{enumerate}

\end{lemma}

\begin{proof}

(a). Let $g: [0,1] \rightarrow S^1 \cong ([0,1]/ 0 \sim 1)$ be the quotient map, then $g^* \gamma^*E$ is a torus bundle over $[0,1]$, which is trivial. Moreover, we could assume that the homeomorphism $\Psi : g^* \gamma^*E \rightarrow [0,1] \times T^2$ preserves the additive structure on each fiber, that is, $\Psi$ restricted on the fiber of $g^* \gamma^*E$ at point $t$, denoted as $\Psi_t$, is a group isomorphism from the fiber to $T^2 \cong (\mathbb{R}/\mathbb{Z})^2$. Then $\Psi_0, \Psi_1$ are two isomorphisms from $(C_F,p)$ to $T^2$, and  $\gamma^*E$ is homeomorphic to 

$$T_f = ([0,1] \times T^2) /\{ (0, v) \sim (1, f(v)) \, | \, v \in T^2 \}$$
where $f = \Psi_1 \circ \Psi_0^{-1} $. For convenience, we also denote the homeomorphism from $\gamma^*E$ to $T_f$ as $\Psi$.

The action of $f$ on the homology group $f_*: H_1(T^2;\mathbb{Z}) \rightarrow H_1(T^2; \mathbb{Z})$ is equal to the monodromy action of $\gamma$. Because $\rho(\gamma)=-I$, we have $f_*=-I$. Moreover $f$ is a self-isomorphism of $T^2$, so $f$ must map $v$ to $-v$.

\vspace{1mm}

(b). Every section  of $\gamma^*E$ corresponds to a section $ [0,1] \rightarrow g^* \gamma^*E$ whose image at 0 and 1 are the same. The monodromy action of $\gamma$ on $H_1(C_F; \mathbb{Z}/6\mathbb{Z})$ is also $-I$, which implies that for a sextatic point $q \in C_F$, there exists a section of $g^* \gamma^*E$ that starts from $(0,q)$ and ends at $(1,-q)$, in which every point is a sextatic point. There are exactly 3 sextatic points $q$ that satisfy $q=-q$, so there are exactly 3 sections of the bundle $\gamma^*E$ in which every point is a sextatic point. 

\vspace{1mm}

(c). Let $\bar{\gamma} : S^1 \rightarrow \gamma^*E$ be one of the three sections given in (b). Then its base point is $(F,q)$, where $q$ is a 2-torsion.

Let $\Psi$ be the homeomorphism we have found in the proof of (a). $\Psi$ preserves the additive structure on each fiber, so it maps a sextatic point to $(t, v) \in T_f$ where $6v=(0,0) \in T^2$ and $3v \neq (0,0)$. Therefore, $\Psi$ maps the set of sextatic points to a disjoint union of lines $S^1 \times \{v\}$, where $6v = 0$ and $3v \neq 0$. The base point of $\Psi \circ \gamma$ is $(0,v)$, where $v$ is a 2-torsion, hence it must be one of the lines, so $\Psi(\gamma(t)) = (t,v)$ with $v \in \{ (\frac{1}{2},0), (0, \frac{1}{2}), (\frac{1}{2}, \frac{1}{2}) \}$.

\end{proof}

We are now ready to prove (d) of Proposition \ref{sext}.

\vspace{3mm}

\begin{proof}[Proof of Proposition \ref{sext}, (d)]

Recall that $\Gamma_\mathrm{flex} = \pi_1(\tilde{\mathcal{X}}_\mathrm{flex})/Z(K) \cong \mathrm{SL}_2(\mathbb{Z})$, $\Gamma_\mathrm{sext} = \pi_1(\tilde{\mathcal{X}}_\mathrm{sext}) /Z(K)$ is a subgroup of $\mathrm{SL}_2(\mathbb{Z})$, and $\Lambda = \pi_1(E)/Z(K) \cong \mathbb{Z}^2 \rtimes \mathrm{SL}_2(\mathbb{Z})$. Also, $Z(K)$ lies in the center of both $\pi_1(\tilde{\mathcal{X}}_\mathrm{flex}), \pi_1(\tilde{\mathcal{X}}_\mathrm{sext})$ and $\pi_1(E)$, so $\pi_1(\tilde{\mathcal{X}}_\mathrm{flex}) = \Gamma_\mathrm{flex} \times Z(K) $, $\pi_1(\tilde{\mathcal{X}}_\mathrm{sext}) = \Gamma_\mathrm{sext} \times Z(K)$, and $\pi_1(E) = \Lambda \times Z(K)$.

In $\Lambda$ we have $(\phi(-I) , -I) = (\phi(-I) , I) \cdot (0, -I)$. Take the base point of $E$ to be $(F,p)$, an inflection point. Let $ \gamma_{\mathrm{sext}}, \gamma_C, \gamma_{\mathrm{flex}}$ be loops in $E$, whose homotopy class are $((\phi(-I) , -I),0), ((\phi(-I) , I),0), ((0, -I),0)$ respectively. Then they satisfy the following four conditions:

\begin{enumerate} \setlength{\itemindent}{1em}
    \item [(1).] The embedding of the fiber $C_F$ into $E$ induces an injection $\pi_1(C_F) \rightarrow \Lambda, v \mapsto (v,I)$. In particular, $(\phi(-I) , I)$ is an image of this injection, so $\gamma_C$ is homotopic to a loop completely in $C_F$.
    \item [(2).] The section map $s_{\mathrm{flex}}:  \tilde{\mathcal{X}}_\mathrm{flex} \rightarrow E$ induces an injection $\Gamma_\mathrm{flex} \rightarrow \Lambda, g \mapsto (0 , g)$. In particular, $(0, -I)$ is an image of this injection, so $\gamma_{\mathrm{flex}}$ is homotopic to a loop in which every point is an inflection point in $E$.
    \item [(3).] The section map $s_{\mathrm{sext}}:  \tilde{\mathcal{X}}_\mathrm{sext} \rightarrow E$ induces an injection $\Gamma_\mathrm{sext} \rightarrow \Lambda, g \mapsto (\phi(g) , g) $, where $\phi$ is a cocycle, and is a coboundary if and only if $[s_{\mathrm{sext}}]$ is trivial. In particular, $(\phi(-I), -I)$ is an image of this injection, so $\gamma_{\mathrm{sext}}$ is homotopic to $l_1 \cdot \bar{\gamma} \cdot l_1^{-1}$,  where $\bar{\gamma}$ is a loop in which every point is a sextatic point, $l_1$ is a path in $C_F$ from $(F,p)$ to the base point of $\bar{\gamma}$, and $l_1^{-1}$ is the inverse of $l_1$.
    \item [(4).] $\gamma_{\mathrm{sext}}$ is homotopic to $\gamma_C \cdot \gamma_{\mathrm{flex}}$ since $(\phi(-I) , -I) = (\phi(-I) , I) \cdot (0, -I)$. 
    
\end{enumerate}

Next we show that for any triple of loops $\gamma_C, \gamma_{\mathrm{flex}},  \gamma_{\mathrm{sext}}$ in $E$ that satisfy the four conditions above, then the homotopy class of $\gamma_C$ is an element in $\pi_1(C_F)$ not divisible by 2. Let $\gamma: S^1 \rightarrow \tilde{\mathcal{X}}_{\mathrm{flex}}$ be the image of $\gamma_{\mathrm{sext}}$ under the covering map $f:\tilde{\mathcal{X}}_{\mathrm{sext}} \rightarrow \tilde{\mathcal{X}}_{\mathrm{flex}}$, where the existence of such a covering map is guaranteed by Proposition \ref{covering_map}. Then $\gamma_{\mathrm{sext}} \subset \gamma^* p^*_{\mathrm{flex}} E$. For convenience, we will denote $\gamma^* p_{\mathrm{flex}}^* E$ by $\gamma^*E$ in the remainder of the proof, just as we have done in Lemma \ref{mapping_torus}. We also denote a point in $\gamma^*E$ by its image under $\Psi$, where $\Psi$ is a homeomorphism from  $\gamma^*E$ to the mapping torus
$$T_f = ([0,1] \times T^2) /\{ (0, v) \sim (1, -v) \, | \, v \in T^2 \}$$ given by (c), Lemma \ref{mapping_torus}. 

The base point of $\gamma^* E$ is $(0,(0,0))$, and $ \gamma(t) = (t,(0,0))$. By  (c) in Lemma \ref{mapping_torus}, a loop where every point is a sextatic point can only be $\bar{\gamma}_{v} : S^1 \rightarrow \gamma_{\mathrm{flex}}^*E, \, t \mapsto (t, v)$, where $v$ is a 2-torsion in $T^2$. Therefore, $\gamma_{\mathrm{sext}}$ is homotopic to $l_1 \cdot \bar{\gamma}_v \cdot l_1^{-1}$ for a 2-torsion $v \in T^2$ and a path $l_1$ inside $C_F$ that goes from $(0,(0,0))$ to $v$.

Let $l_1(t) = (0, l(t))$. Define $l_2(t) = (0, -l(1-t)) $, then $l_2(0) = v$ and $l_2(1) =(0,0)$. Let 
$$ H(s,t) =  \begin{cases} 
(0, -l(1-2t)) & 0 \leq t \leq \frac{1}{2}s, \\
(2t, -l(1-2s)) & \frac{1}{2}s \leq t \leq \frac{1}{2}(s+1), \\
(0, l(2-2t)) & \frac{1}{2}(s+1) \leq t \leq 1, \\
\end{cases}
$$
where $s \in [0,1]$ and $t \in S^1 \cong ( [0,1] / 0 \sim 1 )$. Then $H(0,t) \simeq \bar{\gamma}_v  \cdot l_1^{-1}$ and $H(1,t) \simeq l_2 \cdot \gamma$. Therefore, $\gamma_{\mathrm{sext}} \simeq l_1 \cdot \bar{\gamma}_v \cdot l_1^{-1}$ is homotopic $l_1 \cdot l_2 \cdot \gamma$.

\begin{figure}[htp]
    \centering
    \includegraphics[width=10cm]{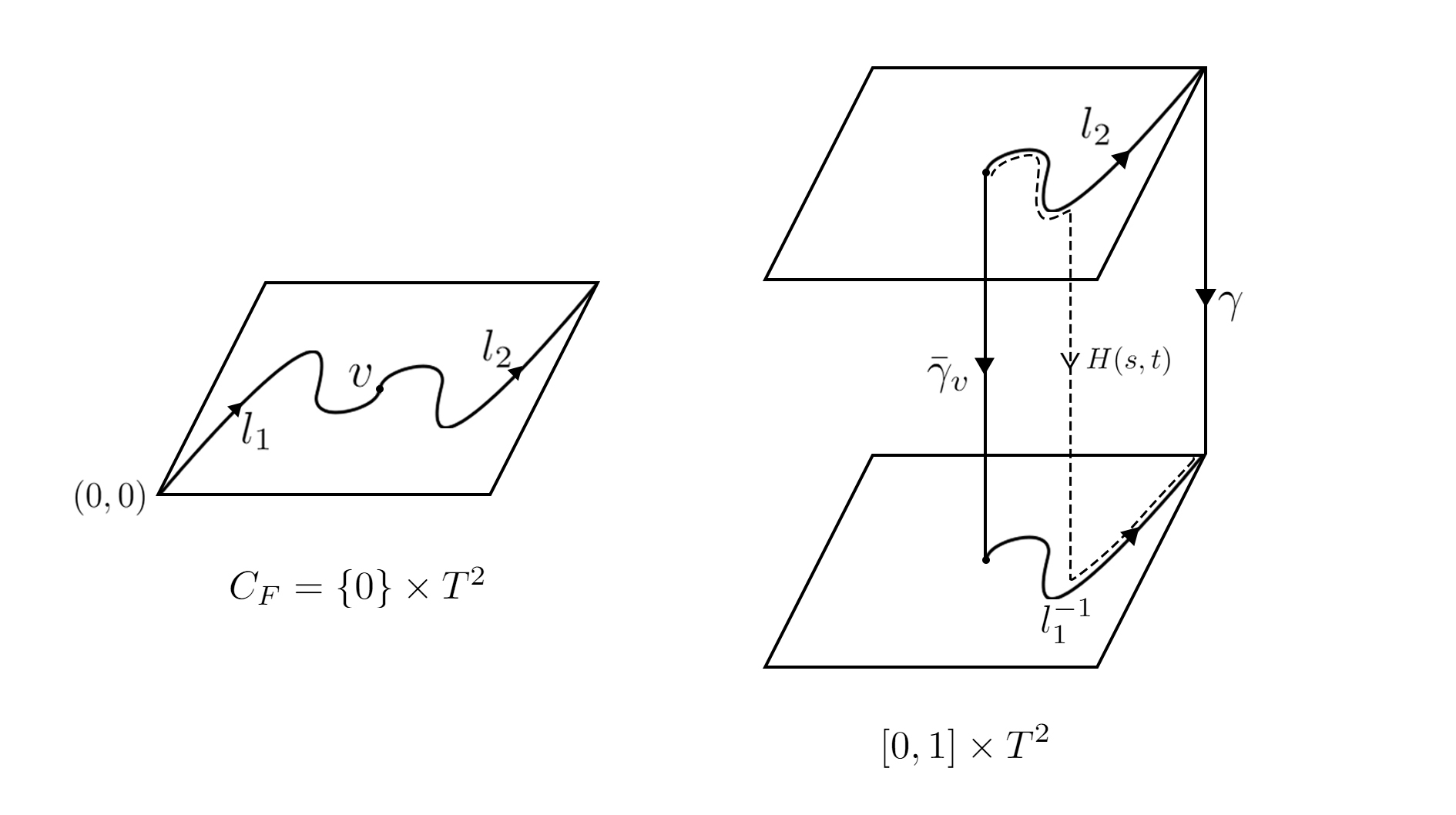}
    \captionsetup{font={footnotesize}}
    \caption{The paths and the homotopy in the proof. $C_F$ is represented by a parallelogram, and $T_f$ is represented by a parallelepiped. (Both of them should have opposite edges glued by some rules, but we omit the gluing rules for the conciseness of the figure.)}
    \label{fig:my_label}
\end{figure}

Furthermore, $l_1 \cdot l_2$ lies in $C_F$ and $\gamma$ is a loop where every point is an inflection point. Since $\Lambda \cong \pi_1(C_F) \rtimes \Gamma_{\mathrm{flex}}$, there is only one way up to homotopy to write any loop in $E$ into a concatenation of a loop in $C_F$ and a loop where every point is an inflection point. Hence, $\gamma_C \simeq l_1 \cdot l_2.$

Lift $l_1 \cdot l_2$ from $T^2$ to its universal cover $\mathbb{R}^2$. Denote the lift of $l_1$ that begins at $(0,0)$ as $\tilde{l}_1$, and assume that it ends at $\tilde{v}$, then $2\tilde{v} \in \mathbb{Z}^2$ and $\tilde{v} \notin \mathbb{Z}^2$. Denote the lift of $l_2$ that begins at $\tilde{v}$ as $\tilde{l}_2$. Recall that $l_1^{-1}(t) = (0, l(1-t))$ and  $l_2(t) = (0, -l(1-t))$. because $\tilde{l}_1^{-1}$ goes from $\tilde{v}$ to $(0,0)$, we know that $\tilde{l}_2$ goes from $\tilde{v}$ to $2\tilde{v}$. Because $\tilde{v} \notin \mathbb{Z}^2$, the end point of $\tilde{l}_1 \cdot \tilde{l}_2$ is an element in $\mathbb{Z}^2$ not divisible by 2, and the homotopy class of $l_1 \cdot l_2$ in $\pi_1(C_F)$ is not divisible by 2.

On the other hand, if $\phi$ is a coboundary, then $\phi(-I) = x - (-I) \cdot x = 2x$ for some $x \in \mathbb{Z}^2$, so $(\phi(-I), I) = (x, I) \cdot (x, I)$ and the homotopy class of $\gamma_C$ must be divisible by 2. Hence, the cocycle $\phi$ is not a coboundary and $[s_{\mathrm{sext}}]$ is the nontrivial element in $H^1(\Gamma_\mathrm{sext}; \mathbb{Z}^2)$.

\end{proof}

\subsection{Proof of Proposition \ref{homotopic_sectionmap}}

We now finish the proof of Proposition 3.1.

\begin{proof}[Proof of Proposition \ref{homotopic_sectionmap}]

(a). By Proposition \ref{covering_map}, we already know that there exists a covering map $f:\tilde{\mathcal{X}}_n \rightarrow \tilde{\mathcal{X}}_\mathrm{flex}$  such that $p_\mathrm{flex} \circ f = p$. 

If $[s_n]$ is trivial, consider two section maps $s_n$ and $s_\mathrm{flex} \circ f$. We claim that $ [s_\mathrm{flex} \circ f ] =0 \in H^1(\Gamma_{n};\mathbb{Z}^2) $. Indeed, $[ s_\mathrm{flex} \circ f ]$ can be seen as the image of $s_\mathrm{flex} \in H^1(\Gamma_\mathrm{flex};\mathbb{Z}^2)$ under $f^*:  H^1(\Gamma_\mathrm{flex};\mathbb{Z}^2) \rightarrow  H^1(\Gamma_n;\mathbb{Z}^2) $, which is zero since $[s_\mathrm{flex}]=0.$ Since $s_n$ is also zero, the two section maps are homotopic.

\vspace{2mm}

(b). If $[s_n]$ is nontrivial, assume that $[s_n]$ is represented by $\phi(g) = \frac{1}{2} (g \cdot v - v)$ for some $v$. If $v \in (2\mathbb{Z})^2$ then $\phi_{\bar{v}}(g)$ is a coboundary, so $v \notin (2\mathbb{Z})^2$. 

Change the basis of $H_1(C_F; \mathbb{Z})$ so that 
$v = \begin{bmatrix}
1 \\ 0
\end{bmatrix}$ mod 2
under the new basis. Then for every $g \in \Gamma_n$, we have $(g-I) \cdot \begin{bmatrix}
1 \\ 0
\end{bmatrix}  \in (2\mathbb{Z})^2$, and if we write $g = \begin{bmatrix}
a & b \\
c & d
\end{bmatrix}$, then  $\begin{bmatrix}
a-1 \\ c
\end{bmatrix} \in (2\mathbb{Z})^2$. Therefore 
$$g = \begin{bmatrix}
1 & 0 \\
0 & 1 
\end{bmatrix}  \mathrm{or} \begin{bmatrix}
1 & 1 \\
0 & 1 
\end{bmatrix} \mathrm{mod} \ 2,$$ 
and by (b) in Proposition \ref{sext}, we have $\Gamma_n \subset \Gamma_\mathrm{sext}$. 

To prove that there is a covering map $f:\tilde{\mathcal{X}}_n \rightarrow \tilde{\mathcal{X}}_\mathrm{sext}$, we only need to prove that $\tilde{\Gamma}_n$ can be conjugated into $\tilde{\Gamma}_\mathrm{sext}$. Again, consider the following split exact sequences, which are similar to the split exact sequence in the proof of Proposition \ref{covering_map}: 

\[ 1 \rightarrow (\mathbb{Z}/3\mathbb{Z})^2 \rightarrow (\mathbb{Z}/3\mathbb{Z})^2 \rtimes \Gamma_n \rightarrow \Gamma_n \rightarrow 1 ,\]
\[ 1 \rightarrow (\mathbb{Z}/3\mathbb{Z})^2 \rightarrow (\mathbb{Z}/3\mathbb{Z})^2 \rtimes \Gamma_{\mathrm{sext}} \rightarrow \Gamma_{\mathrm{sext}} \rightarrow 1, \]
where $\tilde{\Gamma}_n$ is the image of a section of the first exact sequence, and $\Gamma_{\mathrm{sext}}$ is the image of a section of the second exact sequence.
Since  $H^1(\Gamma_n;(\mathbb{Z}/3\mathbb{Z})^2)=0$ for any subgroup $\Gamma_n \subset \mathrm{SL}_2(\mathbb{Z})$ that contains $-I$ (this is proved inside the proof of  Proposition \ref{covering_map}), both sections $$\Gamma_n \rightarrow (\mathbb{Z}/3\mathbb{Z})^2 \rtimes \Gamma_n , \ \  \Gamma_{\mathrm{sext}} \rightarrow (\mathbb{Z}/3\mathbb{Z})^2 \rtimes \Gamma_{\mathrm{sext}}$$ are conjugate to the standard section, and are therefore conjugate to each other when restricted to $\Gamma_n$.
So $\tilde{\Gamma}_n$ can be conjugated into $\tilde{\Gamma}_\mathrm{sext}$.

Finally, consider the two sections $s_n$ and $s_\mathrm{sext} \circ f$. We have that $ [s_\mathrm{sext} \circ f ]$ is the image of $[s_\mathrm{sext}] \in H^1(\Gamma_\mathrm{sext};\mathbb{Z}^2)$ under $f^*:  H^1(\Gamma_{\mathrm{sext}};\mathbb{Z}^2) \rightarrow  H^1(\Gamma_n;\mathbb{Z}^2) $, which is represented by  $\phi(g) = \frac{1}{2} (g \cdot \begin{bmatrix}
1 \\ 0
\end{bmatrix} - \begin{bmatrix}
1 \\ 0
\end{bmatrix})$. This is the same as $[s_\mathrm{sext}] \in H^1(\Gamma_n;\mathbb{Z}^2)$. Therefore, $s_n$ is homotopic to $s_\mathrm{sext} \circ f$.

\end{proof}

\section{Braids over loops in $\tilde{\mathcal{X}}_\mathrm{flex}$ and $\tilde{\mathcal{X}}_\mathrm{sext}$}

In this section, we will prove the following proposition: 

\begin{proposition} \label{m=1}
Any connected virtual section $\tilde{\mathcal{X}}_n$ such that $-I \in \Gamma_n$ is homotopic to $ \tilde{\mathcal{X}}_\mathrm{flex} $ or $ \tilde{\mathcal{X}}_\mathrm{sext} $.
\end{proposition}

The method to prove this proposition is to use the homotopy we have gained in Proposition \ref{homotopic_sectionmap} to build braids in Artin's braid group $Br_m$ for certain loops in $ \tilde{\mathcal{X}}_\mathrm{flex} $ or $ \tilde{\mathcal{X}}_\mathrm{sext} $. 

In 4.1, we introduce a way to construct braids through the homotopy relation. In 4.2, we prove Proposition \ref{m=1}, and also list the ingredients we need to complete the proof of Proposition \ref{m=1}. In 4.3, we construct a braid over any nullhomotopic loop, while in 4.4, we construct another braid over a certain loop $\gamma$ such that $\gamma \cdot \gamma$ is nullhomotopic and $\gamma$ is not. These two subsections will produce the ingredients to complete the proof of Proposition \ref{m=1}.
\vspace{3mm}

\subsection{Constructing braids over loops using homotopy relation}

Let $\xi_\mathrm{flex}$ (resp. $\xi_\mathrm{sext}$) be the vector bundle over $\tilde{\mathcal{X}}_\mathrm{flex}$ (resp. $\tilde{\mathcal{X}}_\mathrm{sext}$), whose fiber over a point $(F,p)$ (resp. $(F,q)$) is the tangent space $T_pC_F$ (resp. $T_qC_F$). They are complex line bundles since $C_F$ is a Riemann surface. Any complex line bundle over a loop is trivial, so we can consider global trivializations of $\gamma^* \xi_{\mathrm{flex}}$ or $\gamma^* \xi_{\mathrm{sext}}$ for any loop $\gamma$.

Recall that in Proposition \ref{homotopic_sectionmap}, a virtual section $\tilde{\mathcal{X}}_n$ such that $-I \in \Gamma_n$ satisfies either the conditions in (a) or the conditions in (b). For the remainder of this paper, if a connected virtual section $( \tilde{\mathcal{X}}_n,p_n,s_n )$ satisfies conditions in (a) of Proposition \ref{homotopic_sectionmap}, then we say it is an \textit{inflection-cased} virtual section, and if it satisfies conditions in (b) of Proposition \ref{homotopic_sectionmap}, we say it is a \textit{sextatic-cased} virtual section.

We now claim the following lemma.

\begin{lemma} \label{braid}
Suppose $(\tilde{\mathcal{X}}_n, p_n, s_n)$ is an inflection-cased virtual section. 
Fix a base point $(F_0, p_0)$ in $\tilde{\mathcal{X}}_\mathrm{flex}$, and fix an isomorphism $h: T_{p_0}C_{F_0} \rightarrow \mathbb{C}$.

Given a loop $\gamma: S^1 \rightarrow \tilde{\mathcal{X}}_\mathrm{flex}$ with base point $(F_0, p_0)$, and a trivialization $\mathcal{T} : \gamma^* \xi_{\mathrm{flex}} \rightarrow \mathbb{C} \times S^1$ such that its restriction on the fiber at point 0 is $h: T_{p_0}C_{F_0} \rightarrow \mathbb{C}$. Then there is a braid $b(\gamma, \mathcal{T}) \in Br_m$, where $m = \frac{n}{9}$, satisfying the following conditions:

\begin{enumerate} \setlength{\itemindent}{1em}
    \item [(1).] If $(\gamma_0, \mathcal{T}_0)$ is based homotopic to  $(\gamma_1, \mathcal{T}_1)$, then $b(\gamma_0, \mathcal{T}_0) = b(\gamma_1, \mathcal{T}_1)$.
    
    \item [(2).] $b(\gamma_1 \cdot \gamma_2, \mathcal{T}_1 \cdot \mathcal{T}_2) = b(\gamma_1, \mathcal{T}_1) \cdot b(\gamma_2, \mathcal{T}_2)$. 
    
    \item [(3).] $b(\gamma, e^{2\pi i t} \mathcal{T} ) = \Delta^2 \cdot b(\gamma,  \mathcal{T} )$. Here the trivialization $e^{2\pi i t} \mathcal{T}$ is defined as $(e^{2\pi i t} \mathcal{T} ) (t, V) = (t, e^{2\pi i t} \cdot V)$, and $\Delta^2 = ( \sigma_1 \sigma_2 ... \sigma_{m-1})^m$, where $\sigma_i$ are the Artin generators of $Br_m$.
\end{enumerate}

\end{lemma}

Both condition (1) and (2) need some clarification. In the first condition, since $\mathcal{T}$ depends on $\gamma$, the homotopy of $\gamma$ and the homotopy of $\mathcal{T}$ cannot be treated separately. We will give our definition here:

\begin{definition} \label{trivialization_homotopy}
The pairs of loops and trivializations, $(\gamma_0, \mathcal{T}_0)$ and $(\gamma_1, \mathcal{T}_1)$, are \textit{based homotopic} if the following properties hold:
\begin{enumerate} \setlength{\itemindent}{1em}
    \item [(1).] There is a based homotopy $H: [0,1] \times S^1 \rightarrow \tilde{\mathcal{X}}_{\mathrm{flex}}$ such that $H(0,t) = \gamma_0(t)$ and $H(1,t) = \gamma_1(t)$;
    \item [(2).] There is a based homotopy $T : H^*\xi_{\mathrm{flex}} \rightarrow \mathbb{C} \times S^1$ such that 
    $$T|_{ \{0 \} \times S^1} = \mathcal{T}_0, \  T|_{ \{1 \} \times S^1} = \mathcal{T}_1,$$
    and $T|_{ \{s \} \times S^1}$ is a trivialization of $\gamma_s^* \xi_{\mathrm{flex}}$ for every $s \in [0,1]$, where $\gamma_s(t)=H(s,t)$. It is equivalent to saying $T$ maps the fiber at point $(s,t)$ to $\mathbb{C} \times \{t\}$.
    \item [(3).] For every $s \in [0,1]$, the trivialization $T|_{ \{s \} \times S^1}$ restricted to the fiber at base point $(s,0)$ is fixed, in other words, $T|_{(s,0)} = h, \forall s \in [0,1]$ for some isomorphism $h: T_{p_0}C_{F_0} \rightarrow \mathbb{C}$.
\end{enumerate}

\end{definition}

In the second condition of Lemma \ref{braid}, $\gamma_1 \cdot \gamma_2$ is the concatenation of two loops and $b(\gamma_1, \mathcal{T}_1) \cdot b(\gamma_2, \mathcal{T}_2)$ is the product of two braids. We define the notation $\mathcal{T}_1 \cdot \mathcal{T}_2$:

\begin{definition}
If we explicitly write
$$ \gamma_1 \cdot \gamma_2: S^1  \rightarrow \tilde{\mathcal{X}}_\mathrm{flex},$$
$$t \mapsto \begin{cases}
\gamma_1(2t), & t \in [0, \frac{1}{2}], \\
\gamma_2(2t-1), & t \in [\frac{1}{2}, 1],
\end{cases}$$
then under this definition, $(\gamma_1 \cdot \gamma_2)^* (\xi_\mathrm{flex})$ is a vector bundle over $S^1$ whose fiber at point $t$ is the fiber of $\gamma_1^* (\xi_\mathrm{flex})$ at point $2t$ if $t \in [0, \frac{1}{2}]$, and the fiber of $\gamma_2^* (\xi_\mathrm{flex})$ at point $2t-1$ if $t \in [\frac{1}{2}, 1]$. Define

$$\mathcal{T}_1 \cdot \mathcal{T}_2: (\gamma_1 \cdot \gamma_2)^* (\xi_\mathrm{flex}) \rightarrow \mathbb{C} \times S^1, $$
$$ (t, V) \mapsto \begin{cases}
    \mathcal{T}_1(2t,V), & t \in [0, \frac{1}{2}],  \\
 \mathcal{T}_2(2t-1,V), & t \in [\frac{1}{2}, 1]. 
\end{cases} $$

\end{definition}

We are now ready to prove Lemma \ref{braid}.

\begin{proof}[Proof of Lemma \ref{braid}]

\vspace{2mm}

The proof of Lemma \ref{braid} will have the following structure: first we construct a loop $\Omega(\gamma, \mathcal{T})$ inside $\mathrm{UConf}_m(\mathbb{C})$ for each loop $\gamma$ and trivialization $\mathcal{T}$, then we show that the homotopy class of $\Omega(\gamma, \mathcal{T})$, which is an element in $\pi_1 (\mathrm{UConf}_m(\mathbb{C}) ) \cong Br_m $, satisfies the conditions stated in Lemma \ref{braid}.

\vspace{2mm}

We construct the loop $\Omega(\gamma, \mathcal{T})$ by the following steps: 

Step 1. For each point $(F,p) \in \tilde{\mathcal{X}}_\mathrm{flex}$, the preimage $f^{-1}(F,p)$ are $m$ distinct points, denoted as $X_1, \dots, X_m$, where $m=\frac{n}{9}$.
Since $s_n$ is injective, $s_n$ maps these $m$ points to distinct points in $E$, denoted as $(F,x_1), \dots, (F,x_m)$ respectively, where $x_i \in C_F$ for $i=1, \dots, m$. We have $(s_\mathrm{flex} \circ f)(X_i) = (F,p)$, and $s_n(X_i) =(F, x_i)$. Let $s_t$ be the homotopy such that $s_0 = s_\mathrm{flex} \circ f$ and $s_1 = s_n$. Then $\gamma_i(t) = s_t(F,x_i)$ defines a path from $p$ to $x_i$ inside $C_F$.

Step 2. Now we equip $C_F$ with its unique flat Riemannian metric compatible with its complex structure. Then every path $\gamma_i$ in $C_F$ is homotopic to a unique geodesic segment $\theta_i(\tau)$ where $\tau \in [0,1] $. Take 
$$V_i = \frac{\mathrm{d}}{\mathrm{d}\tau} \theta_i(\tau) |_{\tau=0}, $$
then $V_i$ are distinct for $i=1, \dots, n$. Therefore we can define a map $\tilde{s}_n : \tilde{\mathcal{X}}_n \rightarrow \xi_\mathrm{flex}$ by $\tilde{s}_n(X_i)= (F, V_i)$. 

Step 3. $\tilde{s}_n(\tilde{\mathcal{X}}_n)$ is a subset of $\xi_\mathrm{flex}$ which intersects every fiber $T_pC_F$ with $m$ distinct points. Therefore for any loop $\gamma : S^1 \rightarrow \tilde{\mathcal{X}}_\mathrm{flex}$, the pullback bundle $\gamma^* \xi_\mathrm{flex}$ contains a subset $ \gamma^* \tilde{s}_n(\tilde{\mathcal{X}}_n)$, which intersects every fiber with $m$ distinct points.

Step 4. Any complex line bundle over a circle is trivial. Therefore for any loop $\gamma : S^1 \rightarrow \tilde{\mathcal{X}}_\mathrm{flex}$, the pullback bundle $\gamma^* \xi_\mathrm{flex}$ is trivial. Given a trivialization $\mathcal{T}: \gamma^* \xi_\mathrm{flex} \rightarrow  \mathbb{C} \times S^1 $, its image $\mathcal{T}( \gamma^* \tilde{s}_n(\tilde{\mathcal{X}}_n) )$ is a subset of $\mathbb{C} \times S^1$, which intersects $\mathbb{C} \times \{t\}$ at $m$ points for any $t \in S^1$. 
Define $\Omega(\gamma, \mathcal{T})$ to be
$$S^1 \rightarrow \mathrm{UConf}_m(\mathbb{C}), \ t \mapsto \mathcal{T} ( \gamma^* \tilde{s}_n(\tilde{\mathcal{X}}_n) )  \cap (\mathbb{C} \times \{t\} ).$$ 

Notice that since
$$\Omega(\gamma, \mathcal{T})(0) = \mathcal{T} ( \gamma^* \tilde{s}_n(\tilde{\mathcal{X}}_n) )  \cap (\mathbb{C} \times \{0\} ) = \mathcal{T}(  \tilde{s}_n(\tilde{\mathcal{X}}_n) \cap T_{p_0}C_{F_0} ),$$
and $\mathcal{T}$ restricted to the fiber $T_{p_0}C_{F_0}$ is a fixed isomorphism $h: T_{p_0}C_{F_0} \rightarrow \mathbb{C}$, we know that $\Omega(\gamma, \mathcal{T})(0)$ is a point in $\mathrm{UConf}_m(\mathbb{C})$ that does not depend on $\gamma$ or $\mathcal{T}$. We call this point $A$ for the rest of the proof. Let $b(\gamma,\mathcal{T})$ be the homotopy class of $\Omega(\gamma,\mathcal{T})$ in $\pi_1(\mathrm{UConf}_m(\mathbb{C}), A) \cong Br_m$.

\vspace{1mm}

Next, we show that $b(\gamma,\mathcal{T})$ satisfies (1) in Lemma \ref{braid}.
Suppose $(H,T)$ is a based homotopy between the pair of loops and trivializations as in  Definition 4.1, then each $T|_{\{s\} \times S^1 }$ is a trivialization of $\gamma_s^*\xi_{\mathrm{flex}}$, where $\gamma_s(t)=H(s,t)$. Take 
$$ H' : [0,1] \times S^1 \rightarrow \mathrm{UConf}_m(\mathbb{C}),$$
$$ (s,t) \mapsto T|_{\{s\} \times S^1 } \, ( \gamma_s^* \tilde{s}_n(\tilde{\mathcal{X}}_n) )  \cap ( \mathbb{C} \times \{t\} ). $$

Then we have
$$H'(0,t)=  \mathcal{T}_0( \gamma_0^* \tilde{s}_n(\tilde{\mathcal{X}}_n) ) \cap ( \mathbb{C} \times \{t\} ) = \Omega(\gamma_0, \mathcal{T}_0)(t), $$
$$H'(1,t)=  \mathcal{T}_1( \gamma_1^* \tilde{s}_n(\tilde{\mathcal{X}}_n) ) \cap ( \mathbb{C} \times \{t\} ) = \Omega(\gamma_1, \mathcal{T}_1)(t). $$ 

So $\Omega(\gamma_0, \mathcal{T}_0)$ and $\Omega(\gamma_1, \mathcal{T}_1)$ are homotopic and $b(\gamma_0, \mathcal{T}_0) = b(\gamma_1, \mathcal{T}_1).$

\vspace{2mm}
Then, we show that $b(\gamma,\mathcal{T})$ satisfies (2) in Lemma \ref{braid}. By checking the definition of $\Omega(\gamma, \mathcal{T})$ in Step 4, and also Definition 4.2, one could find that 

$$ \Omega(\gamma_1 \cdot \gamma_2, \mathcal{T}_1 \cdot \mathcal{T}_2)(t) =  \begin{cases}
\Omega(\gamma_1, \mathcal{T}_1 )(2t) & t \in [0, \frac{1}{2}], \\
\Omega(\gamma_2, \mathcal{T}_2 )(2t-1) & t \in [\frac{1}{2},1].
\end{cases}$$

So $\Omega(\gamma_1 \cdot \gamma_2, \mathcal{T}_1 \cdot \mathcal{T}_2) = \Omega(\gamma_1, \mathcal{T}_1) \cdot \Omega(\gamma_2, \mathcal{T}_2)$, and $b(\gamma_1 \cdot \gamma_2, \mathcal{T}_1 \cdot \mathcal{T}_2) = b(\gamma_1, \mathcal{T}_1) \cdot b(\gamma_2, \mathcal{T}_2)$.

Finally, we show that $b(\gamma,\mathcal{T})$ satisfies (3) in Lemma \ref{braid}. Indeed, we have $\Omega(\gamma, e^{2 \pi it} \mathcal{T})(t) = e^{2 \pi i t} \cdot \Omega(\gamma,  \mathcal{T})$, where $e^{2 \pi i t} \cdot X = \{ e^{2 \pi i t} z  \, | z \in X \}$ for any $X \in \mathrm{UConf}_m(\mathbb{C})$. Consider the following homotopy 

$$H : [0,1] \times S^1 \rightarrow \mathrm{UConf}_m(\mathbb{C}),$$
$$ (s,t) \mapsto e^{2\pi i \min\{ (1+s)t,1 \}} \cdot \Omega(\gamma,  \mathcal{T}) ( \max \{ 0, (1+s)t - s \} ),$$
where $t$ are treated as in $[0,1)$ in this definition. Then the loop $H_0$ is $ \Omega(\gamma, e^{2 \pi it} \mathcal{T})$ and the loop $H_1 $ is $ \omega \cdot \Omega(\gamma,  \mathcal{T})$, where $\omega: S^1 \rightarrow \mathrm{UConf}_m(\mathbb{C})$ is a loop defined as $\omega(t) = e^{2\pi i t} \cdot A$. It is known that the homotopy class of $\omega$ is $\Delta^2 \in Br_m$. Therefore, $b(\gamma, e^{2\pi i t} \mathcal{T} ) = \Delta^2 \cdot b(\gamma,  \mathcal{T})$.

\end{proof}

For sextatic-cased virtual sections, we have a lemma similar to Lemma \ref{braid}:

\begin{lemma} \label{braid, only sextatic}
Suppose $(\tilde{\mathcal{X}}_n, p_n, s_n)$ is a sextatic-cased virtual section. 
Fix a base point $(F_0, q_0)$ in $\tilde{\mathcal{X}}_\mathrm{sext}$, and fix an isomorphism $h: T_{q_0}C_{F_0} \rightarrow \mathbb{C}$.

Given a loop $\gamma: S^1 \rightarrow \tilde{\mathcal{X}}_\mathrm{sext}$ with base point $(F_0, q_0)$, and a trivialization $\mathcal{T} : \gamma^* \xi_{\mathrm{sext}} \rightarrow \mathbb{C} \times S^1$ such that its restriction on the fiber at point 0 is $h: T_{q_0}C_{F_0} \rightarrow \mathbb{C}$. Then there is a braid $b(\gamma, \mathcal{T}) \in Br_m$, where $m = \frac{n}{27}$, satisfying the following conditions:

\begin{enumerate} \setlength{\itemindent}{1em}
    \item [(1).] If $(\gamma_0, \mathcal{T}_0)$ is based homotopic to  $(\gamma_1, \mathcal{T}_1)$, then $b(\gamma_0, \mathcal{T}_0) = b(\gamma_1, \mathcal{T}_1)$.
    
    \item [(2).] $b(\gamma_1 \cdot \gamma_2, \mathcal{T}_1 \cdot \mathcal{T}_2) = b(\gamma_1, \mathcal{T}_1) \cdot b(\gamma_2, \mathcal{T}_2)$. 
    
    \item [(3).] $b(\gamma, e^{2\pi i t} \mathcal{T} ) = \Delta^2 \cdot b(\gamma,  \mathcal{T} )$. 
\end{enumerate}

\end{lemma}

The proof is also similar to the proof of Lemma \ref{braid}, where we only need to replace all $\tilde{\mathcal{X}}_\mathrm{flex}$, $s_\mathrm{flex}$, $\xi_\mathrm{flex}$ with $\tilde{\mathcal{X}}_\mathrm{sext}$, $s_\mathrm{sext}$, $\xi_\mathrm{sext}$ respectively, and note that for the sextatic case $m = \frac{n}{27}$.

\subsection{Proof of Proposition \ref{m=1}}

Before our proof of Proposition \ref{m=1}, we need the following lemma:

\begin{lemma} \label{trivial_braid}
Let $( \tilde{\mathcal{X}}_n, p_n, s_n) $ be an inflection-cased virtual section.
Suppose $b(\gamma_0, \mathcal{T}_0) \in Br_m$ is the braid that satisfy the conditions in Proposition \ref{braid}.

\begin{enumerate} \setlength{\itemindent}{1em}
    \item [(a).] For any nullhomotopic loop $\gamma_0$ in $\tilde{\mathcal{X}}_{\mathrm{flex}}$ with base point $(F_0,p_0)$, there exists a trivialization $\mathcal{T}_0$ of $\gamma_0^* \xi_{\mathrm{flex}}$ such that $b(\gamma_0, \mathcal{T}_0)=1 \in Br_m$.
    \item [(b).] There exists a loop $\gamma$ in $\tilde{\mathcal{X}}_{\mathrm{flex}}$ with base point $(F_0,p_0)$ such that $\gamma \cdot \gamma$ is a nullhomotopic loop while $\gamma$ is not.  Its monodromy action on $H_1(C_F; \mathbb{Z})$ is $-I$.
    \item [(c).] There exists a trivialization $\mathcal{T}'$ of $\gamma^* \xi_{\mathrm{flex}}$, such that $ \mathcal{T}' \cdot \mathcal{T}' \simeq e^{2 \pi  it} \mathcal{T}_0 $ as trivializations of $(\gamma \cdot \gamma)^* \xi_{\mathrm{flex}}$, where $\mathcal{T}_0$ is the trivialization decided by (a). Here, $ e^{2 \pi  it} \mathcal{T}_0$ is defined as $( e^{2 \pi  it} \mathcal{T}_0 )(t, V)=  \mathcal{T}_0 (t,  e^{2 \pi  it} \cdot V). $
    \item [(d).] Assume $\gamma$ and $\mathcal{T}'$ have the same definitions as in (b) and (c). Then $b(\gamma, \mathcal{T}') \in Br_m$ is a pure braid.
\end{enumerate}

\end{lemma}

(a) in Lemma \ref{trivial_braid} will be proved in 4.3. (b), (c), and (d) will be proved in 4.4.  

\vspace{2mm}

Lemma \ref{trivial_braid} has a sextatic version:

\begin{lemma} \label{trivial_braid, only sextatic}

Let $( \tilde{\mathcal{X}}_n, p_n, s_n) $ be a sextatic-cased virtual section.
Suppose $b(\gamma_0, \mathcal{T}_0) \in Br_m$ is the braid that satisfy the conditions in Proposition \ref{braid, only sextatic}.

\begin{enumerate} \setlength{\itemindent}{1em}
    \item [(a).] For any nullhomotopic loop $\gamma_0$ in $\tilde{\mathcal{X}}_{\mathrm{sext}}$ with base point $(F_0,q_0)$, there exists a trivialization $\mathcal{T}_0$ of $\gamma_0^* \xi_{\mathrm{sext}}$ such that $b(\gamma_0, \mathcal{T}_0)=1 \in Br_m$.
    \item [(b).] There exists a loop $\gamma$ in $\tilde{\mathcal{X}}_{\mathrm{sext}}$ with base point $(F_0,q_0)$ such that $\gamma \cdot \gamma$ is a nullhomotopic loop while $\gamma$ is not. Its monodromy action on $H_1(C_F; \mathbb{Z})$ is $-I$.
    \item [(c).] There exists a trivialization $\mathcal{T}'$ of $\gamma^* \xi_{\mathrm{sext}}$, such that $ \mathcal{T}' \cdot \mathcal{T}' \simeq e^{2 \pi  it} \mathcal{T}_0 $ as trivializations of $(\gamma \cdot \gamma)^* \xi_{\mathrm{sext}}$, where $\mathcal{T}_0$ is the trivialization decided by (a). Here, $ e^{2 \pi  it} \mathcal{T}_0$ is defined as $( e^{2 \pi  it} \mathcal{T}_0 )(t, V)=  \mathcal{T}_0 (t,  e^{2 \pi  it} \cdot V). $
    \item [(d).] Assume $\gamma$ and $\mathcal{T}'$ have the same definitions as in (b) and (c). Then $b(\gamma, \mathcal{T}') \in Br_m$ is a pure braid.
\end{enumerate}

\end{lemma}

We will not separate the proof of Lemma \ref{trivial_braid, only sextatic} from the proof of Lemma \ref{trivial_braid}. Instead, when we prove (a), (b), (c), (d) in Lemma \ref{trivial_braid} later, we will prove their corresponding sextatic versions at the same time. 

\vspace{3mm}

We now prove Proposition \ref{m=1} using Lemma \ref{trivial_braid} and \ref{trivial_braid, only sextatic}.

\begin{proof}[Proof of Proposition \ref{m=1}]

Recall that connected virtual sections $\tilde{\mathcal{X}}_n$ such that $-I \in \Gamma_n$ can be divided into inflection-cased virtual sections and sextatic-cased virtual sections. We first consider the inflection-cased virtual sections.

Let $\gamma$ be a loop satisfying the conditions in (b) in Lemma \ref{braid}.
Consider $b(\gamma \cdot \gamma, \mathcal{T}' \cdot \mathcal{T}')$. On the one hand, $b(\gamma \cdot \gamma, \mathcal{T}' \cdot \mathcal{T}')$ is equal to $b(\gamma, \mathcal{T}') \cdot b(\gamma, \mathcal{T}')$ by the second condition in Lemma \ref{braid}. On the other hand, by (c) in Lemma \ref{trivial_braid}, $b(\gamma \cdot \gamma, \mathcal{T}' \cdot \mathcal{T}')$ is also equal to $b(\gamma \cdot \gamma, e^{2 \pi  it}  \mathcal{T}_0 )$. By (a) in Lemma \ref{trivial_braid}, we know that $b(\gamma \cdot \gamma, \mathcal{T}_0 )$ is a trivial braid. By the third condition in Lemma \ref{braid}, $b(\gamma \cdot \gamma, e^{2 \pi  it}  \mathcal{T}_0 ) = \Delta^2 \cdot b(\gamma \cdot \gamma,  \mathcal{T}_0 ) = \Delta^2$. 

To summarize, we have $b(\gamma, \mathcal{T}')^2 = \Delta^2$. By Theorem 1.1 of \cite{Braid_root}, any braid $b$ satisfying $b^2 = \Delta^2$ is a conjugate of $\Delta$. For $m \geq 2$, $\Delta$ is not a pure braid, and neither is $b(\gamma, \mathcal{T}')$, because a conjugate of a pure braid is still a pure braid. However, by (d) in Lemma \ref{trivial_braid}, $b(\gamma, \mathcal{T}')$ should be a pure braid. Hence $m=1$, and any inflection-cased virtual section $\tilde{\mathcal{X}}_n$ is homotopic to $\tilde{\mathcal{X}}_{\mathrm{flex}}$. 

For sextatic-cased virtual sections, the proof is similar: Using Lemma \ref{trivial_braid, only sextatic}, by going through the exact same procedures, we also have a pure braid $b(\gamma, \mathcal{T}')$ such that $b(\gamma, \mathcal{T}')^2 = \Delta^2$, hence $m=1$, and $\tilde{\mathcal{X}}_n$ is homotopic to $\tilde{\mathcal{X}}_{\mathrm{sext}}$.

\end{proof}

In the next two subsections, we will prove Lemma \ref{trivial_braid} and \ref{trivial_braid, only sextatic} to complete the proof of Proposition \ref{m=1}. The key is to find a trivialization that satisfies the required conditions.

\subsection{Proofs of (a), Lemma \ref{trivial_braid} and \ref{trivial_braid, only sextatic}}

Let $(F_0,p_0)$ (resp. $(F_0,q_0)$) be the base point of $\tilde{\mathcal{X}}_{\mathrm{flex}}$ (resp. $\tilde{\mathcal{X}}_{\mathrm{sext}}$).
Suppose $\gamma_0 : S^1 \rightarrow \tilde{\mathcal{X}}_{\mathrm{flex}}$ is a nullhomotopic loop through a homotopy $H : [0,1] \times S^1 \rightarrow S^1$. Then $H^*\xi_{\mathrm{flex}}$ is a bundle over $[0,1] \times S^1$ such that 

$$H^*\xi_{\mathrm{flex}} |_{ \{ 0 \} \times S^1} = \gamma_0^*\xi_{\mathrm{flex}}, \ H^*\xi_{\mathrm{flex}} |_{ \{ 1 \} \times S^1} = T_{p_0}C_{F_0} \times [0,1]. $$

We want to define a trivialization of $\gamma_0^*\xi_{\mathrm{flex}}$ using $H^*\xi_{\mathrm{flex}}$. We start by constructing a nonvanishing section $\sigma: [0,1] \times S^1 \rightarrow  H^*\xi_{\mathrm{flex}}$ such that  $\sigma(s,0)  =\sigma (1,t)$ in $T_{p_0}C_{F_0}$ for any $s,t \in [0,1]$.

\begin{lemma}
There exists a nonvanishing section $\sigma: [0,1] \times S^1 \rightarrow  H^*\xi_{\mathrm{flex}}$ (resp. $H^*\xi_{\mathrm{sext}}$) such that  $\sigma(s,0)  =\sigma (1,t)$ in $T_{p_0}C_{F_0}$ (resp. $T_{q_0}C_{F_0}$) for any $s,t \in [0,1]$. 
\end{lemma}

\begin{proof}

Such a section can be constructed by the following procedures:

\begin{enumerate}
    \item Let $g: [0,1] \times [0,1] \rightarrow [0,1] \times S^1 \cong [0,1] \times ([0,1]/\{0 \sim 1 \}) $ be the quotient map.
    Consider any nonvanishing section $\sigma': [0,1] \times [0,1] \rightarrow  g^*H^*\xi_{\mathrm{flex}}$ (resp. $g^*H^*\xi_{\mathrm{sext}}$). 
    \item Since $\xi_{\mathrm{flex}}$ (resp. $\xi_{\mathrm{sext}}$) is a complex line bundle, there is a complex scalar multiplication structure on each fiber of $g^*H^*\xi_{\mathrm{flex}}$ (resp. $g^*H^*\xi_{\mathrm{sext}}$). For any $(s,t)$ such that either $s=1$ or $t=0 \ \mathrm{or} \ 1$, the fiber at point $(s,t)$ is $T_{p_0}C_{F_0}$ (resp. $T_{q_0}C_{F_0}$), and  there is a unique complex number $z(s,t)$ such that $z(s,t) \cdot \sigma'(s,t) = \sigma'(0,0)$. Then $z$ is a continuous function on $([0,1] \times \{0,1\}) \cup (\{1\} \times [0,1])$.
    \item Let $Z : [0,1] \times [0,1] \rightarrow \mathbb{C}$ be any continuous function that satisfies \newline  $Z |_{([0,1] \times \{0,1\}) \cup (\{1\} \times [0,1]) } = z$, and let $\sigma(s,t)=Z(s,t) \cdot \sigma'(s,t)$. Then $\sigma$ satisfies the condition $\sigma(s,0) = \sigma(s,1) =\sigma (1,t)$ for any $s,t \in [0,1]$. 
    \item $\sigma(s,0) = \sigma(s,1) $ guarantees that $\sigma$ gives a well defined  section of $H^*\xi_{\mathrm{flex}}$ (resp. $H^*\xi_{\mathrm{sext}}$). We also denote it by $\sigma$.
\end{enumerate}

\end{proof}

We are now ready to prove (a) in Lemma \ref{trivial_braid} and \ref{trivial_braid, only sextatic}.

\begin{proof}[Proofs of (a), Lemma \ref{trivial_braid} and \ref{trivial_braid, only sextatic}]

Consider (a) in Lemma \ref{trivial_braid} first.
By Lemma 4.6, there exists a section $\sigma$ of $H^*\xi_{\mathrm{flex}}$ such that $\sigma(s,0) = \sigma (1,t)$. Define 

$$ \mathbf{U} : H^* \xi_{\mathrm{flex}} \rightarrow [0,1] \times T_{p_0}C_{F_0},$$ 
$$((s,t), z \cdot \sigma(s,t)) \mapsto  ((1,t), z \cdot \sigma(1,t)), z \in \mathbb{C},$$
where $z \cdot \sigma(0,t)$ and $z \cdot \sigma(1,t)$ are decided by the complex structure on fibers. Since $H^*\xi_{\mathrm{flex}} |_{ \{ 0 \} \times S^1} = \gamma_0^*\xi_{\mathrm{flex}} $, we know that $\mathbf{U}|_{ \{0\} \times S^1}$ is actually a trivialization of $\gamma_0^*\xi_{\mathrm{flex}}$. Denote $\mathcal{T}_s = \mathbf{U} |_{ \{s\} \times S^1}$.  We will show that $\mathcal{T}_0$ satisfies $b(\gamma_0, \mathcal{T}_0) =1$. 

Let $\gamma_s(t) = H(s,t)$. Then each $\mathcal{T}_s$ is a trivialization of $\gamma_s^* \xi_{\mathrm{flex}}$. Because $\sigma(s,0)$ stays constant when $s$ changes, $\mathbf{U} |_{(s,0)}$ is a constant map on $T_{p_0}C_{F_0}$ for any $s \in [0,1]$. Therefore, the pair $(H, \mathbf{U})$ is a based homotopy from $(\gamma_0, \mathcal{T}_0)$ to $(\gamma_1, \mathcal{T}_1)$ (the definition of based homotopy of the pair $(\gamma, \mathcal{T})$ can be found in Definition \ref{trivialization_homotopy}). By Lemma \ref{braid}, This based homotopy implies that $b(\gamma_0, \mathcal{T}_0)=b(\gamma_1, \mathcal{T}_1)$.

It remains to show that $b(\gamma_1, \mathcal{T}_1) $ is the identity in $Br_m$.
Indeed, because $\gamma_1$ is a constant loop, $\gamma_1^*( \tilde{s}_n(\tilde{\mathcal{X}}_n))$ is $A \times S^1 \subset T_{p_0}C_{F_0} \times S^1$, where $A$ is the intersection between $\tilde{s}_n(\tilde{\mathcal{X}}_n)$ and $T_{p_0}C_{F_0}$. Since $ \mathcal{T}_1$ is the identity map, $\mathcal{T}_1 (\gamma_1^*( \tilde{s}_n(\tilde{\mathcal{X}}_n)) ) = A \times S^1 $  and $\Omega(\gamma_1 ,\mathcal{T}_1)$ is a constant loop at point $A$. Therefore $b(\gamma_1 ,\mathcal{T}_1)=1$.

The proof of (a), Lemma \ref{trivial_braid, only sextatic} is almost the same, where we only need to replace all the $\xi_{\mathrm{flex}}$ with $\xi_{\mathrm{sext}}$.

\end{proof}

As a remark, we see that as long as we go through the process listed in the proof above, $b(\gamma_1 ,\mathcal{T}_1)=1$ no matter what the choices of $H$ and $\sigma$ are. 

\subsection{Proofs of (b), (c), (d), Lemma \ref{trivial_braid} and Lemma \ref{trivial_braid, only sextatic}}

In this subsection, we focus on a loop $\gamma$ in $\tilde{\mathcal{X}}_{\mathrm{flex}}$, such that $\gamma \cdot \gamma$ is a nullhomotopic loop, and $\gamma$ is not. We first need to show that such $\gamma$ exists, which is exactly (b) in Lemma \ref{trivial_braid} and \ref{trivial_braid, only sextatic}.

\begin{proof}[Proofs of (b), Lemma \ref{trivial_braid} and \ref{trivial_braid, only sextatic}]
By Corollary 3.12 in \cite{B_Chen}, the monodromy action $\rho$ maps  $\pi_1( \tilde{\mathcal{X}}_{\mathrm{flex}} )/ Z(K)$ to $\Gamma_{\mathrm{flex}}$ and  $\pi_1( \tilde{\mathcal{X}}_{\mathrm{sext}} )/ Z(K)$ to $\Gamma_{\mathrm{sext}}$, where $\Gamma_{\mathrm{flex}}$ and $\Gamma_{\mathrm{sext}}$ are subgroups of $\mathrm{SL}_2(\mathbb{Z})$, and $Z(K)$ is a group isomorphic to $\mathbb{Z}/ 3\mathbb{Z}$. By Lemma 6.2 in \cite{B_Chen} (or see Proposition \ref{flex} and \ref{sext}), both $\Gamma_{\mathrm{flex}}$ and $\Gamma_{\mathrm{sext}}$ contain $-I$.

By Equation (4.9) in \cite{Dolgachev}, $Z(K)$ is the center of $\pi_1(\mathcal{X})$, so $Z(K)$ lies in the center of $\pi_1( \tilde{\mathcal{X}}_{\mathrm{flex}})$ or $\pi_1( \tilde{\mathcal{X}}_{\mathrm{sext}})$, since both are subgroups of $\pi_1(\mathcal{X})$. Therefore $\pi_1( \tilde{\mathcal{X}}_{\mathrm{flex}} ) = \Gamma_{\mathrm{flex}} \times Z(K) $ and $\pi_1( \tilde{\mathcal{X}}_{\mathrm{sext}} ) = \Gamma_{\mathrm{sext}} \times Z(K) $, so $(-I, 0)$ is a nontrivial element in both $\pi_1( \tilde{\mathcal{X}}_{\mathrm{flex}} )$ and $\pi_1( \tilde{\mathcal{X}}_{\mathrm{sext}} )$. Also $(-I, 0) \cdot (-I, 0) = (I, 0)$. Take $\gamma$ to be a loop in this homotopy class, and it satisfies the conditions in (b), Lemma \ref{trivial_braid} and \ref{trivial_braid, only sextatic}.

\end{proof}

Because $\gamma \cdot \gamma$ is nullhomotopic, we already know a trivialization of $(\gamma \cdot \gamma)^*\xi_{\mathrm{flex}}$ from (a), Lemma \ref{trivial_braid} or \ref{trivial_braid, only sextatic}. In this subsection, we aim to find a trivialization of $\gamma^*\xi_{\mathrm{flex}}$ starting from the trivialization of $(\gamma \cdot \gamma)^*\xi_{\mathrm{flex}}$. 

Let $H: [0,1] \times S^1 \rightarrow \tilde{\mathcal{X}}_{\mathrm{flex}}$ be a homotopy from $\gamma \cdot \gamma$ to the constant loop at $(F_0,p_0)$. Denote $\gamma_s(t)=H(s,t)$, then $\gamma_0 = \gamma \cdot \gamma$ and $\gamma_1 = id_{(F_0,p_0)}$. Recall that in 4.3, we use a section $\sigma$ of $H^* \xi_{\mathrm{flex}}$ to construct $\mathcal{T}_0$. Our first goal is the following lemma:

\begin{lemma} \label{sigma_section}
Let  $\gamma : S^1 \rightarrow \tilde{\mathcal{X}}_{\mathrm{flex}}$ (resp. $\tilde{\mathcal{X}}_{\mathrm{sext}}$) be a loop such that $\gamma \cdot \gamma$ is nullhomotopic and $\gamma$ is not, and $H: [0,1] \times S^1 \rightarrow \tilde{\mathcal{X}}_{\mathrm{flex}}$ (resp. $\tilde{\mathcal{X}}_{\mathrm{sext}}$) be a homotopy from $\gamma \cdot \gamma$ to the constant loop at $(F_0,p_0)$ (resp. $(F_0,q_0)$).

Then there exists a section $\sigma : [0,1] \times S^1 \rightarrow H^* \xi_{\mathrm{flex}}$ (resp. $H^* \xi_{\mathrm{sext}}$), such that $\sigma(0, t+\frac{1}{2}) = - \sigma(0, t)$ for any $t \in S^1$.
\end{lemma}

The equation $\sigma(0, t+\frac{1}{2}) = - \sigma(0, t)$ makes sense: Since $H(0, t+\frac{1}{2}) = H(0,t)$, the fibers of $H^*\xi_{\mathrm{flex}}$ (resp. $H^*\xi_{\mathrm{sext}}$) at $(0, t+\frac{1}{2})$ and $(0, t)$ are the same.

\vspace{2mm}

The idea of the proof of Lemma \ref{sigma_section} is as follows: First, consider the pullback bundle $H^* (p_{\mathrm{flex}}^*E)$, where $p_{\mathrm{flex}}^*E$ is a bundle of $\tilde{\mathcal{X}}_{\mathrm{flex}}$, whose fiber at point $(F,p)$ is $C_F$. For each $(s,t ) \in [0,1] \times S^1$, we build a loop $\theta_{s,t}$ inside the respecting torus $C_F$ for the point $H(s,t) \in \tilde{\mathcal{X}}_{\mathrm{flex}}$. Then we homotope each $\{ \theta_{s,t} \}$ to a geodesic on the torus, and take $\sigma(s,t)$ as the velocity of the geodesic at the base point. We will prove that $\sigma(0, t+\frac{1}{2}) = - \sigma(0, t)$ using the monodromy action of $\gamma$.

The following lemma is needed in our proof to make sure that $\{ \theta_{s,t} \}$ is a continuous family of loops with respect to $s$ and $t$. 

\begin{lemma} \label{torus_trivial}
Suppose $\gamma_0$ is a nullhomotopic loop with base point $(F_0,p_0)$ (resp. $(F_0, q_0)$), and $H: [0,1] \times S^1 \rightarrow \tilde{\mathcal{X}}_{\mathrm{flex}}$ (resp. $\tilde{\mathcal{X}}_{\mathrm{sext}}$) is a homotopy from $\gamma_0$ to the constant loop at $(F_0,p_0)$ (resp. $(F_0, q_0)$). Then we have a trivialization 
$$ \mathbf{U}: H^* (p_{\mathrm{flex}}^*E) \rightarrow T^2 \times [0,1] \times S^1,$$
such that $\mathbf{U}((s,t), i ( H(s,t) ) ) =  ((0,0), s,t)$, where
$$i: \tilde{\mathcal{X}}_{\mathrm{flex}} \rightarrow p_{\mathrm{flex}}^*E, \, (F,p) \mapsto ((F,p), p)$$ 
is the natural immersion. 

For the sextatic case, we have a similar trivialization $\mathbf{U}$ from $H^* (p_{\mathrm{sext}}^*E)$ to $T^2 \times [0,1] \times S^1$.
\end{lemma}

\begin{proof}
We construct a trivialization of $H^*(p_{\mathrm{flex}}^*E)$ by the following procedures:

\begin{enumerate}
    \item  Let $g: [0,1] \times [0,1] \rightarrow [0,1] \times S^1 = [0,1] \times ([0,1]/\{0 \sim 1 \}) $ be the quotient map. Then $ g^*H^*(p_{\mathrm{flex}}^*E)$ is a torus bundle over $[0,1] \times [0,1]$, which is trivial since $[0,1] \times [0,1]$ is contractible.
    \item Let $\mathbf{U}'$ be any trivialization of $g^*H^*(p_{\mathrm{flex}}^*E)$, and 
    $$\mathbf{U}(s,t, (H(s,t), x)) = \mathbf{U}'(s,t, (H(s,t), x)) - \mathbf{U}'(s,t, i(H(s,t)) ) $$
    for any point $x$ in the fiber of $p_{\mathrm{flex}}^*E$ at point $H(s,t)$. 
    Then $\mathbf{U}$ is also a trivialization of $g^*H^*(p_{\mathrm{flex}}^*E)$, and
    $$\mathbf{U}((s,t), i ( H(s,t) ) ) = \mathbf{U}'(s,t, i(H(s,t)) ) - \mathbf{U}'(s,t, i(U(s,t)) ) =  ((0,0), s,t).$$
    \item The fiber of $g^*H^*(p_{\mathrm{flex}}^*E)$ at point $(s,t)$ is $C_{F_0}$ for any $(s,t) \in ([0,1] \times \{0,1\} ) \cup (\{1\} \times [0,1])$. We have that $\{ \mathbf{U}|_{(s,t)} \}$ is a family of homeomorphisms from $C_{F_0}$ to $T^2$, which changes continuously when $(s,t)$ moves on $([0,1] \times \{0,1\} ) \cup (\{1\} \times [0,1])$. 
    \item Since $([0,1] \times \{0,1\} ) \cup (\{1\} \times [0,1])$ is homeomorphic to $[0,1]$, we can homotope $\mathbf{U}|_{([0,1] \times \{0,1\} ) \cup (\{1\} \times [0,1])}$ so that $\mathbf{U}|_{(s,t)} = \mathbf{U}|_{(0,0)}$ for any $ (s,t) \in  ([0,1] \times \{0,1\} ) \cup (\{1\} \times [0,1])$. We can extend this homotopy to the whole $\mathbf{U}$, since $([0,1] \times [0,1] )/  ([0,1] \times \{0,1\} ) \cup (\{1\} \times [0,1])$ is contractible.
    \item Since $\mathbf{U}|_{(s,0)} = \mathbf{U}|_{(s,1)}$ for any $s \in [0,1]$, the trivialization $\mathbf{U}$ gives a well-defined trivialization of $H^*(p_{\mathrm{flex}}^*E)$, which we also denote as $\mathbf{U}$.
    
\end{enumerate}

The construction of the trivialization of $H^*(p_{\mathrm{sext}}^*E)$ follows the similar procedures, where we only need to replace all the $p_{\mathrm{flex}}^*E$ by $p_{\mathrm{sext}}^*E$. 
\end{proof}

The following results from Riemannian geometry will also be useful:

\begin{lemma} \label{geodesic}
Suppose $(C_F, p)$ is a torus with a base point, equipped with its unique flat Riemannian metric compatible with its complex structure. Then:

(a). Any loop with base point $p$ is based homotopic to a unique geodesic that begins and ends at $p$. If two loops are homotopic, then they are homotopic to the same geodesic. By taking the velocity of the geodesic at time 0, we have a well-defined map $P: \pi_1(C_F, p) \rightarrow T_pC_F$.

(b). $P(- [\theta]) = - P( [\theta])$ for any $[\theta] \in \pi_1(C_F, p)$. 
\end{lemma}

\begin{proof}
Let $\mathbb{C} \rightarrow C_F$ be the universal covering of $C_F$, and take any uplift $\tilde{p}$ of $p$. Then every loop $\theta$ in $C_F$ is uniquely lifted to a path in $\mathbb{C}$ that begins at $\tilde{p}$, and its endpoint is decided by the homotopy class of $\theta$. Equip $\mathbb{C}$ with the flat Riemannian metric, then a loop is a geodesic if and only if its lift in $\mathbb{R}^2$ is also a geodesic, in other words, a straight line. Any path in $\mathbb{C}$ is based homotopic to a unique straight line, so we proved (a). 

Now assume $\theta_1, \theta_2$ are two loops in $C_F$ with base point $p$, such that $[\theta_1] = -[\theta_2]$. Suppose the lift of $\theta_i$ has end point $\tilde{p}_i$ for $i=1,2$, then $\tilde{p}_1 - \tilde{p} = -( \tilde{p}_2 - \tilde{p})$, and $P(- [\theta]) = - P([\theta])$. 

\end{proof}

We will prove Lemma \ref{sigma_section} using Lemma \ref{torus_trivial} and \ref{geodesic}.

\begin{proof}[Proof of Lemma \ref{sigma_section}]

Let $\mathbf{U}$ be the trivialization of $H^*(p_{\mathrm{flex}}^*E)$ we have found in Lemma \ref{torus_trivial}. Take $\theta_{s,t}(\tau) = \mathbf{U}^{-1}((\tau, 0), s,t)$. Then $\{ \theta_{s,t} \}$ is a continuous (with respect to $s$ and $t$) family of loops in $H^* (p_{\mathrm{flex}}^*E)$, satisfying the following conditions:

\begin{enumerate}
    \item $\theta_{s,t}$ is a loop that lies completely in the fiber of bundle $\gamma_s^* (p_{\mathrm{flex}}^*E)$ at point $H(s,t) = \gamma_s(t)$. 
    \item $\theta_{s,t}(0) = i(H(s,t))$. Here $i: \tilde{\mathcal{X}}_{\mathrm{flex}} \rightarrow p_{\mathrm{flex}}^*E, (F,p) \mapsto ((F,p), p)$ 
is the natural immersion. 
    
\end{enumerate} 

By (a) in Lemma \ref{geodesic}, for any $(s,t) \in [0,1] \times S^1$, there is a homotopy between $\theta_{s,t}$ and a unique geodesic. Let $H_{s,t,\tau}$ be the homotopy such that $H_{s,t,0} = \theta_{s,t}$ and $H_{s,t,1} $ is a geodesic  $\theta'_{s,t}$. Then $\theta'_{s,t}$ is well-defined, and $\{ \theta'_{s,t} \}$ is a continuous family of loops. Let $\sigma_{s,t}$ be the velocity of $\theta'_{s,t}$ at point $\tau=0$, then $\sigma_{s,t}$ is well-defined, nonzero, and continuous by $s$ and $t$. Therefore, $\sigma$ is a nonvanishing section of $H^* \xi_{\mathrm{flex}}$.

\vspace{2mm}

We next show that $\sigma(0, \frac{1}{2}) = - \sigma(0,0).$ Recall that by the definition the monodromy action, if let $$\bar{\gamma}: [0,1] \rightarrow \tilde{\mathcal{X}}_{\mathrm{flex}},\ t \mapsto \begin{cases}
\gamma(t) &  t \in [0,1), \\
\gamma(0) &  t=1,
\end{cases}$$
and there is a continuous family of loops $\alpha_t$ inside $\bar{\gamma}^* (p_{\mathrm{flex}}^*E)$, satisfying that each  $\alpha_t$ lies in the fiber of $\bar{\gamma}^* (p_{\mathrm{flex}}^*E)$ at point $t$, then the homology classes of $\alpha_0$ and $\alpha_1$, denoted as $[\alpha_0]$ and $[\alpha_1]$, satisfy that $[\alpha_1] = \rho(\gamma_*) [\alpha_0]$, where $\rho$ is the monodromy representation $\pi_1(\tilde{\mathcal{X}}_{\mathrm{flex}}) \rightarrow H_1(C_{F_0}; \mathbb{Z})$. Now take $\alpha_t = \theta'_{0,\frac{1}{2}t}$. Since $\rho([\gamma]) = -I$, we have $ [ \theta'_{0,\frac{1}{2}} ]  = - [\theta'_{0,0} ] $ as homology classes. Since $\pi_1(C_F) \cong H_1(C_F ; \mathbb{Z})$, for homotopy classes we also have $[\theta'_{0,\frac{1}{2}}]  = -[\theta'_{0,0}] $. Therefore, by Lemma \ref{geodesic}, (b),  $\sigma(0, \frac{1}{2}) = - \sigma(0,0).$

\vspace{2mm}

Finally, we show that $\sigma(0, t + \frac{1}{2}) = - \sigma(0, t)$ for any $t \in S^1.$  Consider the following loop $$\tilde{\gamma}(t') = (\gamma \cdot \gamma) (t +\frac{1}{2}t'),$$ 
whose base point is not $(F_0,p_0)$, but rather $\gamma(2t)$. Nevertheless, the homotopy class of $\tilde{\gamma}$ in  $\pi_1( \tilde{\mathcal{X}}_{\mathrm{flex}}, \gamma(2t) )$, denoted  $[\tilde{\gamma}]$, is the image of $[\gamma]$ under the isomorphism $\pi_1( \tilde{\mathcal{X}}_{\mathrm{flex}}, (F_0,p_0) ) \rightarrow  \pi_1( \tilde{\mathcal{X}}_{\mathrm{flex}}, \gamma(2t) ) $, so $\tilde{\gamma}$ also have monodromy action $-I$. By the same arguments in the previous paragraph, we know that $\sigma(0, t + \frac{1}{2}) = - \sigma(0, t) $ for any $ t \in S^1.$ 

The proof for the sextatic case is similar, where we only need to replace all the $\tilde{\mathcal{X}}_{\mathrm{flex}}$ by $\tilde{\mathcal{X}}_{\mathrm{sext}}$.
\end{proof}

\vspace{2mm}

We are now ready to prove (c) and (d) in Lemma \ref{trivial_braid} and \ref{trivial_braid, only sextatic}.

\begin{proof}[Proofs of (c), Lemma \ref{trivial_braid} and \ref{trivial_braid, only sextatic}]

We prove (c) in Lemma \ref{trivial_braid} first.

Let $\mathcal{T}_0 : (\gamma \cdot \gamma)^*\xi_{\mathrm{flex}} \rightarrow T_{p_0}C_{F_0} \times S^1 $  be 

$$ ((0,t), z \cdot \sigma(0,t)) \mapsto  ((1,t), z \cdot \sigma(1,t)), z \in \mathbb{C} ,$$
and  $\mathcal{T}' : (\gamma \cdot \gamma)^*\xi_{\mathrm{flex}} \rightarrow T_{p_0}C_{F_0} \times S^1 $  be  

$$ ((0,t), z \cdot \sigma(0,t)) \mapsto  ((1,t), e^{2\pi i t} z \cdot \sigma(1,t)),   z \in \mathbb{C},$$
where $\sigma$ is the section given by Lemma \ref{sigma_section}.

By (a) in Lemma \ref{trivial_braid},  $b(\gamma \cdot \gamma, \mathcal{T}_0 ) = 1$. On the other hand, $\mathcal{T}' ((0,t), z \cdot \sigma(0,t))  = \mathcal{T}' ((0,t+\frac{1}{2}), z \cdot \sigma(0,t+\frac{1}{2}))$, so $\mathcal{T}'$ restricted to $t \in [0, \frac{1}{2}]$ is a trivialization of $\gamma^* \xi_{\mathrm{flex}}$. For convenience, we also denote it as $\mathcal{T}'$. Then we have $\mathcal{T}' \cdot \mathcal{T}' \simeq e^{2\pi i t} \mathcal{T}_0$. 

The proof of (c), Lemma \ref{trivial_braid, only sextatic} is almost the same, where we only need to replace every $ \xi_{\mathrm{flex}}$ with $\xi_{\mathrm{sext}}$. 

\end{proof}

\vspace{2mm}

\begin{proof}[Proofs of (d), Lemma \ref{trivial_braid} and Lemma \ref{trivial_braid, only sextatic}]

We prove (d) in Lemma \ref{trivial_braid} first.
Recall that in the proof of Lemma \ref{braid}, the definition of $\Omega(\gamma,\mathcal{T}')$ is
$$ \Omega(\gamma, \mathcal{T}')( t ) = \mathcal{T'} ( \gamma^* \tilde{s}_n(\tilde{\mathcal{X}}_n) )  \cap (\mathbb{C} \times \{t\} ).$$

To prove that $b(\gamma,\mathcal{T}' )$ is a pure braid, it means to prove that  $\mathcal{T}'( \gamma^* \tilde{s}_n(\tilde{\mathcal{X}}_n) )$ consists of $m$ many disjoint loops, and since $\mathcal{T}'$ is a homeomorphism, it is equivalent to prove that $\gamma^* \tilde{s}_n(\tilde{\mathcal{X}}_n) $ consists of $m$ many disjoint loops.

Recall that $\tilde{\mathcal{X}}_n$ is an $m$-sheeted covering space of $\tilde{\mathcal{X}}_{\mathrm{flex}}$, and so is $\tilde{s}_n(\tilde{\mathcal{X}}_n)$ since $\tilde{s}_n$ is injective. Therefore, $ \pi_1(\tilde{s}_n(\tilde{\mathcal{X}}_n)) \cong \pi_1( \tilde{\mathcal{X}}_n) $ is a subgroup of $ \pi_1( \tilde{\mathcal{X}}_{\mathrm{flex}} ) $. Again, let $(F_0,p_0)$ be the base point of $\tilde{\mathcal{X}}_{\mathrm{flex}}$. There are $m$ many points, $(F_0,V_1), \dots,  (F_0,V_m),$ that project to $(F_0,p_0)$ through the covering map.
For any $i=1, \dots, m $, the loop $\gamma$ can be uniquely lifted to a path in $\tilde{s}_n(\tilde{\mathcal{X}}_n)$ that starts from $(F_0,x_i)$. To prove that $ \gamma^* \tilde{s}_n(\tilde{\mathcal{X}}_n)$ consists of $m$ many loops, we only need to show that every lift of $\gamma$ into $\tilde{s}_n(\tilde{\mathcal{X}}_n)$ is also a loop.

The loop $\alpha$ which starts from $\tilde{s}_n (F_0,p_0)$ is lifted to a loop if and only if $[\alpha] \in \pi_1( \tilde{\mathcal{X}}_n )$, if and only if $\rho([\alpha]) \in \Gamma_n$. The lift of two loops $\alpha_1, \alpha_2$ has the same ending point if and only if $\alpha_1 \cdot \alpha_2^{-1} $ is lifted to a loop, if and only if $\Gamma_n \rho( [\alpha_1] )$ and $\Gamma_n \rho( [\alpha_2] )$ are the same coset of $\Gamma_n \subset \Gamma_{\mathrm{flex}}$. Therefore, each point $(F_0,V_i)$ uniquely corresponds to a coset. Because the monodromy action of the loop $\gamma$ is $-I$, the monodromy action of the loop $\gamma$ on $\{(F,V_i)| i=1,2, \dots, m\}$ maps a coset $ \Gamma_n \rho([\alpha]) $ to $ ( \Gamma_n  \rho( [\alpha] ) ) \cdot (-I)$. Since $-I$ commutes with any $\rho([\alpha])$ and $-I \in \Gamma_n$, two cosets $ \Gamma_n \rho([\alpha]) $ and $ ( \Gamma_n \rho( [\alpha] ) ) \cdot (-I)$ are the same. Therefore, any lift of $\gamma$ into $\tilde{s}_n(\tilde{\mathcal{X}}_n )$ is a loop. This shows that $b(\gamma,\mathcal{T}' )$ is a pure braid.

The proof of (d) in Lemma \ref{trivial_braid, only sextatic} is almost the same, where we only need to replace all $\tilde{\mathcal{X}}_{\mathrm{flex}}$ and $\gamma_{\mathrm{flex}}$ with $\tilde{\mathcal{X}}_{\mathrm{sext}}$ and $\gamma_{\mathrm{sext}}$ respectively.

\end{proof}

Now that we have proved Lemma \ref{trivial_braid} and \ref{trivial_braid, only sextatic}, we have completed the proof of Proposition \ref{m=1}.

\section{Theorems on nonconnected virtual sections}

In this section, we will extend Proposition \ref{m=1} to the general case, where we do not require virtual sections to be connected. We will show that:

\begin{proposition} \label{one_copy}
Suppose $\tilde{\mathcal{X}}_n$ is a possibly nonconnected virtual section. $\tilde{\mathcal{X}}_n$ cannot contain more than one connected component that is homotopic to  $\tilde{\mathcal{X}}_{\mathrm{flex}}$, or more than one connected component that is homotopic to  $\tilde{\mathcal{X}}_{\mathrm{sext}}$.
\end{proposition} 

The idea of the proof of Proposition \ref{one_copy} is as follows: If there are two connected components, both of which are homotopic to $\tilde{\mathcal{X}}_{\mathrm{flex}}$ or $\tilde{\mathcal{X}}_{\mathrm{sext}}$, then we can construct a nonvanishing section of $\xi_{\mathrm{flex}}$ or $\xi_{\mathrm{sext}}$. The existence of such a section implies that the Euler class of $\xi_{\mathrm{flex}}$ or $\xi_{\mathrm{sext}}$ is zero, yet we can compute that neither $\xi_{\mathrm{flex}}$ nor $\xi_{\mathrm{sext}}$ has zero Euler class, which leads to a contradiction. This method was first used in \cite{B_Chen}. 

In 5.1, we explain how to turn Proposition \ref{one_copy} into a problem of calculating the Euler class of $\xi_{\mathrm{flex}}$ and $\xi_{\mathrm{sext}}$. In 5.2, we show that the Euler class of $\xi_{\mathrm{flex}}$ and $\xi_{\mathrm{sext}}$ are nonzero.

\subsection{Proof of Proposition \ref{one_copy}}

We will prove Proposition \ref{one_copy} using a method that Banerjee-Chen used in \cite{B_Chen} to prove that  virtual sections of degree 18 do not exist. Before the proof of Proposition \ref{one_copy},  we need the following propositions:

\begin{proposition}[Proposition 5.4, \cite{B_Chen}] \label{nonzero_euler_class_flex}
Denote the Euler class of $\xi_{\mathrm{flex}}$ by $e(\xi_{\mathrm{flex}})$. Then $e(\xi_{\mathrm{flex}}) \neq 0$.
\end{proposition}

\begin{proposition} \label{nonzero_euler_class}
Denote the Euler class of $\xi_{\mathrm{sext}}$ by $e(\xi_{\mathrm{sext}})$. Then  $e(\xi_{\mathrm{sext}}) \neq 0$.  
\end{proposition}

Proposition \ref{nonzero_euler_class}  will be proved in Subsection 5.2. Let us first use Proposition \ref{nonzero_euler_class_flex} and \ref{nonzero_euler_class}  to prove Proposition \ref{one_copy}.

\begin{proof}[Proof of Proposition \ref{one_copy}]

Assume that there exist two components in $\tilde{\mathcal{X}}_n$ that are both homotopic to $\tilde{\mathcal{X}}_{\mathrm{flex}}$. Let $\tilde{\mathcal{X}}_{18}$ be the disjoint union of the two components, $f_{18} = f_n |_{\tilde{\mathcal{X}}_{18}} $, and $s_{18} = s_n |_{\tilde{\mathcal{X}}_{18}}$. Let $f: \tilde{\mathcal{X}}_{18} \rightarrow \tilde{\mathcal{X}}_{\mathrm{flex}}$ be the covering map. Then $s_{18}$ restricted to each component is homotopic to $s_{\mathrm{flex}}$, and in general $s_{18}$ is homotopic to $s_{\mathrm{flex}} \circ f$.

For each point $(F,p) \in \tilde{\mathcal{X}}_\mathrm{flex}$, the preimage $f^{-1}(F,p)$ is a set of 2 distinct points $X_1$ and $X_2$. Since $s_{18}$ is injective, $s_{18}$ maps $X_1$ and $X_2$ to two distinct points in $E$, denoted by $(F,x_1) $ and $ (F,x_2)$.  We have $(s_\mathrm{flex} \circ f)(X_i) = (F,p)$, and $s_n(X_i) =(F, x_i) \in C_F$. Let $s_t$ be the homotopy such that $s_0 = s_\mathrm{flex} \circ f$ and $s_1 = s_{18}$. Then $\gamma_i(t) = s_t(F,x_i)$ defines a path from $p$ to $x_i$ inside $C_F$.

Equip $C_F$ with its unique flat Riemannian metric compatible with its complex structure, then every path $\gamma_i$ is homotopic to a unique geodesic segment $\theta_i(t)$ with $t \in [0,1]$. Take 
$$V_i = \frac{\mathrm{d}}{\mathrm{d}\tau} \theta_i(\tau) |_{\tau=0}, i=1,2, $$
then $V_1 \neq V_2$. Define $\tilde{s}_{18} : \tilde{\mathcal{X}}_{18} \rightarrow \xi_\mathrm{flex}$ by $\tilde{s}_n(X_i)= (F, V_i)$, and

$$V' : \mathcal{\tilde{\mathcal{X}}}_{18} \rightarrow f^*\xi_\mathrm{flex}, \ \ X_1 \mapsto V_1 - V_2, \, X_2 \mapsto V_2 - V_1,  $$
then $V'$ is a nonvanishing section of $f^*\xi_\mathrm{flex}$.
Because $\tilde{\mathcal{X}}_{18}$ is a disjoint union of two connected components, we can restrict $V'$ to one component and get a section over $\xi_\mathrm{flex}$. Therefore, the Euler class of $\xi_\mathrm{flex}$ must be zero. However, by Proposition \ref{nonzero_euler_class_flex} we have $e(\xi_{\mathrm{flex}}) \neq 0$, so any virtual section cannot contain two components homotopic to $\tilde{\mathcal{X}}_{\mathrm{flex}}$.

With similar arguments, we can prove that any virtual section cannot contain two components homotopic to $\tilde{\mathcal{X}}_{\mathrm{sext}}$, using $e(\xi_{\mathrm{sext}}) \neq 0$, which is Proposition \ref{nonzero_euler_class}. We only need to replace $\tilde{\mathcal{X}}_{18}$ with $\tilde{\mathcal{X}}_{54}$, a disjoint union of two components homotopic to $\tilde{\mathcal{X}}_{\mathrm{sext}}$, and then replace $p_{18}, s_{18}, \tilde{s}_{18}, \tilde{\mathcal{X}}_{\mathrm{flex}}, \xi_{\mathrm{flex}}$ with $p_{54}, s_{54}, \tilde{s}_{54}, \tilde{\mathcal{X}}_{\mathrm{sext}}, \xi_{\mathrm{sext}}$ respectively.

\end{proof}

In the next section, we will calculate the Euler class of $\xi_\mathrm{sext}$.

\subsection{Proof of  Proposition \ref{nonzero_euler_class}}

We will calculate $\xi_\mathrm{sext}$ using the similar method that Banerjee-Chen calculated $\xi_\mathrm{flex}$ in \cite{B_Chen}.

\vspace{2mm}

Since $\xi_\mathrm{sext}$ is an oriented bundle with real dimension 2, it is a pullback of the universal bundle over $B\mathrm{SL}_2 (\mathbb{R})$ by a classifying map $\phi_\mathrm{sext}$. The universal Euler class is a generator $e$ of $H^2(B \mathrm{SL}_2(\mathbb{R}) ; \mathbb{Z})$, and $e(\xi_{\mathrm{sext}}) = \phi_\mathrm{sext}^* (e)$. To prove $e(\xi_{\mathrm{sext}}) \neq 0$, we will prove a sequence of lemmas that investigate the behavior of $\phi_\mathrm{sext}$.

\begin{lemma} \label{can_decompose}
The classifying map $\phi_\mathrm{sext}$ is homotopic to $\iota \circ i \circ P_{\mathrm{sext}}$, where:

$P_{\mathrm{sext}}:  \tilde{\mathcal{X}}_{\mathrm{sext}} \rightarrow B \Gamma_{\mathrm{sext}}$ satisfies that $(P_{\mathrm{sext}})_* : \pi_1( \tilde{\mathcal{X}}_{\mathrm{sext}}) \rightarrow \Gamma_{\mathrm{sext}}$ coincides with the monodromy representation $\rho: \pi_1(\tilde{\mathcal{X}}_{\mathrm{sext}}) \rightarrow Aut \, H_1(T^2 ; \mathbb{Z}) $;

$i : B\Gamma_{\mathrm{sext}} \rightarrow B \mathrm{SL}_2(\mathbb{Z})$ is induced by the inclusion $\Gamma_{\mathrm{sext}} \rightarrow \mathrm{SL}_2(\mathbb{Z})$;

$\iota : B \mathrm{SL}_2(\mathbb{Z}) \rightarrow B \mathrm{SL}_2(\mathbb{R})$ is induced by the standard inclusion $\mathrm{SL}_2(\mathbb{Z}) \rightarrow  \mathrm{SL}_2(\mathbb{R})$.

\end{lemma}

This lemma is similar to Lemma 5.5 in \cite{B_Chen}.

\begin{proof}
Let $p_{\mathrm{sext} }: \tilde{\mathcal{X}}_\mathrm{sext} \rightarrow \mathcal{X}$ be the covering map of $\tilde{\mathcal{X}}_\mathrm{sext}$, then $p_{\mathrm{sext}}^*E$ is a bundle over $\tilde{\mathcal{X}}_\mathrm{sext}$ whose fiber at point $(F,q)$ is $C_F$.

We first claim that $\phi_{\mathrm{sext}}$ can be decomposed up to homotopy as follows: 

$$ \tilde{\mathcal{X}}_\mathrm{sext} \xrightarrow{\phi_E}  B\mathrm{Diffeo}^+( T^2,q ) \xrightarrow{\tau} B \mathrm{SL}_2(\mathbb{R}), $$
where $B \mathrm{Diffeo}^+( T^2,q )$ is the group of orientation-preserving diffeomorphisms of the torus $T^2$ that fix a point $q \in T^2$. For two maps, $\phi_E$ is the classifying map of the smooth orientable torus bundle $p_{\mathrm{sext}}^*E$, and $\tau$ is induced by the map $ \mathrm{Diffeo}^+( T^2,q ) \rightarrow \mathrm{SL}_2(\mathbb{R})$ obtained by taking the derivative of a diffeomorphism at point $q$. Indeed, we have this decomposition because while the fiber of $p_{\mathrm{sext}}^*E$ at point $(F,q)$ is $C_F$, the fiber of $\xi_{\mathrm{sext}}$ at point $(F,q)$ is exactly the tangent space $T_qC_F$.

By Theorem 1 in \cite{EE}, Diffeo$_0^+ ( T^2,q ) $, the connected component of Diffeo$^+ (T^2, q)$ that contains identity, is contractible. The mapping class group of the torus, which is defined as Diffeo$^+( T^2,q )/$Diffeo$_0^+( T^2,q )$, is $\mathrm{SL}_2(\mathbb{Z})$ (see \textit{e.g.} Theorem 2.5, \cite{Farbbook} for a proof). So we have a homotopy equivalence between $B$Diffeo$^+( T^2,q )$ and $B\mathrm{SL}_2(\mathbb{Z})$. Now the decomposition becomes

$$ \tilde{\mathcal{X}}_\mathrm{sext} \xrightarrow{\psi_{\mathrm{sext}}}  B\mathrm{SL}_2(\mathbb{Z}) \xrightarrow{\iota} B \mathrm{SL}_2(\mathbb{R}). $$

Again by Theorem 2.5 in \cite{Farbbook}, two diffeomorphisms are isotopic if and only if their actions on $H_1(T^2; \mathbb{Z})$ are the same, so the map that maps a diffeomorphism to its action on $H_1(T^2; \mathbb{Z})$ is a homotopy equivalence between $B\mathrm{Diffeo}^+( T^2,q ) $ and $B\mathrm{SL}_2(\mathbb{Z})$. As a result, $(\psi_\mathrm{sext})_* : \pi_1( \tilde{\mathcal{X}}_\mathrm{sext} ) \rightarrow \mathrm{SL}_2(\mathbb{Z}) $ conicides with the monodromy representation $\rho$ of $\pi_1( \tilde{\mathcal{X}}_\mathrm{sext} )$. Moreover, since $\Gamma_{\mathrm{sext}}$ is not the whole $\mathrm{SL}_2(\mathbb{Z})$, we can further decompose $\psi_\mathrm{sext}$ as follows:

$$ \tilde{\mathcal{X}}_\mathrm{sext} \xrightarrow{P_{\mathrm{sext}}} B\Gamma_{\mathrm{sext}} \xrightarrow{i} B\mathrm{SL}_2(\mathbb{Z}),  $$
such that $(P_{\mathrm{sext}})_*: \pi_1(\tilde{\mathcal{X}}_\mathrm{sext}) \rightarrow \Gamma_{\mathrm{sext}}$ coincides with $\rho$, and $i$ is induced by the inclusion $\Gamma_{\mathrm{sext}} \rightarrow \mathrm{SL}_2(\mathbb{Z})$.

Finally, we claim that the composition $\mathrm{SL}_2(\mathbb{Z}) \rightarrow \mathrm{Diffeo}^+ (T^2, q) \rightarrow \mathrm{SL}_2(\mathbb{R})$ is homotopic to the standard inclusion. Indeed, every connected component in $\mathrm{Diffeo}^+ (T^2, q)$ contains a unique element in the form
$$ T^2 \rightarrow T^2, \ \ \begin{bmatrix}
x \\ y
\end{bmatrix} \mapsto G  \begin{bmatrix}
x \\ y
\end{bmatrix},$$ 
where $x,y \in S^1 \cong \mathbb{R}/\mathbb{Z}$ and $G \in \mathrm{SL}_2(\mathbb{Z})$. Its derivative at any point is $G$. Hence, the composition $\mathrm{SL}_2(\mathbb{Z}) \rightarrow \mathrm{Diffeo}^+ (T^2, q) \rightarrow \mathrm{SL}_2(\mathbb{R})$ is the standard inclusion up to isotopy, and $\iota : B\mathrm{SL}_2(\mathbb{Z}) \rightarrow B\mathrm{SL}_2(\mathbb{R})$ is induced by the standard inclusion.

\end{proof}

\begin{lemma}
$P_{\mathrm{sext}}$ induces an injective map $P_{\mathrm{sext}}^* : H^2(B\Gamma_{\mathrm{sext}} ; \mathbb{Z}) \rightarrow H^2(\tilde{\mathcal{X}}_\mathrm{sext}; \mathbb{Z})$.

\end{lemma}

This lemma is similar to Lemma 5.6 in \cite{B_Chen}.

\begin{proof}
Recall that by Corollary 3.13 in \cite{B_Chen} (or see Proposition \ref{sext} in this paper), $$\pi_1(\tilde{\mathcal{X}}_\mathrm{sext}) / Z(K) \cong \Gamma_{\mathrm{sext}},$$
where the isomorphism is the monodromy action $\rho$. 

Let $Y$ be the cover of $\tilde{\mathcal{X}}_\mathrm{sext}$ associated to $Z(K) \subset \pi_1(\tilde{\mathcal{X}}_\mathrm{sext})$. Then $\pi_1(Y) = Z(K)$ and the deck group is isomorphic to $\Gamma_{\mathrm{sext}}$. This cover induces the following bundle: 
$$ Y \rightarrow Y \times_{\Gamma_\mathrm{sext}} E\Gamma_\mathrm{sext} \rightarrow B\Gamma_\mathrm{sext}$$
where the projection map in the bundle is a covering map, and the projection map induces a map on homotopy classes, from $ 
\pi_1( Y \times_{\Gamma_\mathrm{sext}} E\Gamma_\mathrm{sext} ) \rightarrow \pi_1( B\Gamma_\mathrm{sext} )$, which coincides with the monodromy action $\rho$. Moreover, since $E\Gamma_\mathrm{sext}$ is contracible, $Y \times_{\Gamma_\mathrm{sext}} E\Gamma_\mathrm{sext}$ is homotopic equivalent to $Y/\Gamma_\mathrm{sext} = \tilde{\mathcal{X}}_\mathrm{sext}$. Therefore, we have the following homotopy fibration:

$$ Y \rightarrow  \tilde{\mathcal{X}}_\mathrm{sext} \rightarrow B\Gamma_\mathrm{sext} $$

The projection map in this homotopy fibration is homotopic to $P_{\mathrm{sext}}$, since they induce the same map on fundamental groups. Now consider the Serre spectral sequence of this homotopy fibration:

$$ E_2^{p,q} = H^p ( B\mathrm{SL}_2(\mathbb{Z}) ; H^q (Y; \mathbb{Z}) ) \Rightarrow  H^{p+q}  (\tilde{\mathcal{X}}_\mathrm{sext} ; \mathbb{Z})$$

In particular,  $P_{\mathrm{sext}}^* : H^2(B\Gamma_{\mathrm{sext}} ; \mathbb{Z}) \rightarrow H^2(\tilde{\mathcal{X}}_\mathrm{sext}; \mathbb{Z})$ can be decomposed as follows:

$$ H^2 (B\mathrm{SL}_2(\mathbb{Z}) ; \mathbb{Z}) = E_2^{2,0} \twoheadrightarrow E_{\infty}^{2,0} \hookrightarrow H^2(\tilde{\mathcal{X}}_\mathrm{sext}; \mathbb{Z}) $$

Here $E_2^{2,0} \twoheadrightarrow E_{\infty}^{2,0}$ is the edge morphism, and $E_{\infty}^{2,0} \hookrightarrow H^2(\tilde{\mathcal{X}}_\mathrm{sext}; \mathbb{Z})$ is the inclusion. To show that $P_{\mathrm{sext}}^*$ is injective, we only need to show that the edge morphism $E_2^{2,0} \twoheadrightarrow E_{\infty}^{2,0}$ is an isomorphism.

The only possibly nontrivial differential into the $(2,0)$ entry is $d_2 : E_2^{0,1} \rightarrow  E_2^{2,0}$. We have $E_2^{0,1} = H^0(B \mathrm{SL}_2(\mathbb{Z}) ; H^1(Y ; \mathbb{Z}))$, where 
$$ H^1(Y ; \mathbb{Z}) = \mathrm{Hom}(\pi_1(Y); \mathbb{Z}) =  \mathrm{Hom}(\mathbb{Z}/3\mathbb{Z}; \mathbb{Z}) =0,$$ 
so $E_2^{0,1}=0$. Therefore, the edge morphism $E_2^{2,0} \twoheadrightarrow E_{\infty}^{2,0}$ is indeed an isomorphism.

\end{proof}

\begin{lemma}[Lemma 5.7, \cite{B_Chen}]
$\iota$ induces a surjective map $\iota^* : H^2(B\mathrm{SL}_2(\mathbb{R}); \mathbb{Z}) \rightarrow H^2( B\mathrm{SL}_2(\mathbb{Z}); \mathbb{Z})$.
\end{lemma}

Finally, we prove  Proposition \ref{nonzero_euler_class}.

\begin{proof}[Proof of Proposition \ref{nonzero_euler_class}]

By Lemma \ref{can_decompose}, the classifying map $\phi_\mathrm{sext}$ is homotopic to $\iota \circ i \circ P_{\mathrm{sext}}$, and $e(\xi_\mathrm{sext}) = P_\mathrm{sext}^*  \circ  i^*  \circ \iota^*(e) $, in which $\iota^*$ is surjective by Lemma 5.4, and $ P_\mathrm{sext}^* $ is injective by Lemma 5.5. 

Consider $i^*: H^2(B\mathrm{SL}_2(\mathbb{Z}) ; \mathbb{Z} ) \rightarrow H^2(B\Gamma_\mathrm{sext} ; \mathbb{Z})$. Since $\Gamma_\mathrm{sext}$ is a index 3 subgroup of $\mathrm{SL}_2(\mathbb{Z})$ (see Propsition \ref{sext}),  $B\Gamma_\mathrm{sext}$ is a 3-fold covering map of $B\mathrm{SL}_2(\mathbb{Z})$  and we have a transfer homomorphism $i^!: H^2(B\Gamma_\mathrm{sext} ; \mathbb{Z}) \rightarrow H^2(B\mathrm{SL}_2(\mathbb{Z}) ; \mathbb{Z} ) $, which satisfies that $i^! \circ i^*$ is a multiplication by 3. 

Now we compute $e(\xi_\mathrm{sext}) = P_\mathrm{sext}^*  \circ  i^*  \circ \iota^*(e)$, in which $e$ is a generator of $H^2( B\mathrm{SL}_2(\mathbb{R}) ; \mathbb{Z}) \cong \mathbb{Z}$. Because $\iota^*$ is surjective, $\iota^*(e)$ is a generator of \newline $H^2(B\mathrm{SL}_2(\mathbb{Z}) ; \mathbb{Z} ) \cong \mathbb{Z}/12\mathbb{Z}$, and is a 12-torsion. Because $i^! \circ i^*$ is a multiplication of 3, $i^*$ maps a 12-torsion to either a 12-torsion or a 4-torsion, so $i^* \iota^*(e)$ has order 4 or 12. Because $P_\mathrm{sext}^*$ is injective, $e(\xi_\mathrm{sext})$ has order 4 or 12, and is nonzero.

\end{proof}

Now that we have proved Proposition \ref{nonzero_euler_class}, we have completed the proof of Proposition \ref{one_copy}.

\section{Proofs of Theorem 1.2, 1.3, 1.4}

Before we prove the theorems listed in the introduction, let us start with a lemma that helps us with the cases where $\Gamma_n$ does not contain $-I$.

\begin{lemma}
Suppose $\Gamma_n$ is an index $m$ subgroup of $\mathrm{SL}_2(\mathbb{Z})$  that does not contain $-I$. Then 4 divides $m$.
\end{lemma}

\begin{proof}
Since the only 2-torsion in $\mathrm{SL}_2(\mathbb{Z})$ is $-I$, any 4-torsion $a$ in $\mathrm{SL}_2(\mathbb{Z})$ has $a^2 = -I$. Therefore, any 4-torsion $a \notin \Gamma_n$ since $-I \notin \Gamma_n$. 

Consider the action of $\mathrm{SL}_2(\mathbb{Z})$ on the cosets of $\Gamma_n$ obtained by left multiplication, then $x \cdot \Gamma_n \rightarrow ax \cdot \Gamma_n$ is a permutation on the cosets. We have $ (a^4x)  \Gamma_n =  x  \Gamma_n$ for any coset $x \Gamma_n$. Furthermore, if $(a^i x)  \Gamma_n =  (a^j x) \Gamma_n$, then $x^{-1}a^{i-j}x \in \Gamma_n$. Since $x^{-1}ax$ is also a 4-torsion, if $x^{-1}a^{i-j}x \in \Gamma_n$, then $i-j$ is divisible by 4. Therefore, each orbit of the permutation on cosets $x  \Gamma_n \mapsto  (ax) \Gamma_n$ is a 4-cycle. We then know that the number of cosets is divisible by 4.
\end{proof}

We also have the following proposition:

\begin{proposition}[Classification of components in $\tilde{\mathcal{X}}_n$] \label{classify}
Any virtual section $\tilde{\mathcal{X}}_n$ is a disjoint union of three virtual sections $\tilde{\mathcal{X}}^N,\tilde{\mathcal{X}}^F,\tilde{\mathcal{X}}^S$, where:

\begin{enumerate}
    \item Every connected component of $\tilde{\mathcal{X}}^N$, denoted by $\tilde{\mathcal{X}}_m$, has $\Gamma_m $ not containing $-I$. The degree of $\tilde{\mathcal{X}}^N$ can be divided by $36$.
    
    \item $\tilde{\mathcal{X}}^F$ is an empty set or a connected component homotopic to $\tilde{\mathcal{X}}_\mathrm{flex}$. The degree of $\tilde{\mathcal{X}}^F$ is $0$ or $9$.
    
    \item $\tilde{\mathcal{X}}^S$ is an empty set or a connected component homotopic to $\tilde{\mathcal{X}}_\mathrm{sext}$. The degree of $\tilde{\mathcal{X}}^S$ is $0$ or $27$.
\end{enumerate}

\end{proposition}

\begin{proof}
Let $\tilde{\mathcal{X}}^N$ be the union of all connected components $\tilde{\mathcal{X}}_{9k}$ of $\tilde{\mathcal{X}}_n$ whose $\Gamma_{9k}$ does not contain $-I$;

$\tilde{\mathcal{X}}^F$ be the union of all connected components $\tilde{\mathcal{X}}_{9k}$ of $\tilde{\mathcal{X}}_n$ whose $\Gamma_{9k}$ contains $-I$, and $[s_{9k}] = 0 \in H^1(\tilde{\mathcal{X}}_{9k}; \mathbb{Z}^2)$;

$\tilde{\mathcal{X}}^S$ be the union of all connected components $\tilde{\mathcal{X}}_{9k}$ of $\tilde{\mathcal{X}}_n$ whose $\Gamma_{9k}$ contains $-I$, and $[s_{9k}] \neq 0 \in H^1(\tilde{\mathcal{X}}_{9k}; \mathbb{Z}^2)$.

The degree of $\tilde{\mathcal{X}}^N$ is divisible by 36 from Lemma 6.1.

$\tilde{\mathcal{X}}^F$ is a union of inflection-cased virtual sections. By Proposition \ref{m=1}, each component is homotopic to $\tilde{\mathcal{X}}_{\mathrm{flex}}$. By Proposition \ref{one_copy}, $\tilde{\mathcal{X}}^F$ can contain either 0 or 1 component. Hence, $\tilde{\mathcal{X}}^F$ is either an empty set or homotopic to $\tilde{\mathcal{X}}_{\mathrm{flex}}$.

$\tilde{\mathcal{X}}^S$ is a union of sextatic-cased virtual sections. By Proposition \ref{m=1}, each component is homotopic to $\tilde{\mathcal{X}}_{\mathrm{sext}}$. By Proposition \ref{one_copy}, $\tilde{\mathcal{X}}^S$ can contain either 0 or 1 component. Hence, $\tilde{\mathcal{X}}^S$ is either an empty set or homotopic to $\tilde{\mathcal{X}}_{\mathrm{sext}}$.

\end{proof}

Now we are ready to prove the theorems listed in the introduction. 

\vspace{2mm}

\begin{proof}[Proof of Theorem 1.2]
By Proposition \ref{classify} we know that any virtual section with degree 27 is equal to its $\tilde{\mathcal{X}}^S$, so it must be homotopic to $\tilde{\mathcal{X}}_{\mathrm{sext}}$.
\end{proof}

\begin{proof}[Proof of Theorem 1.3]
If $\Gamma_n$ contains $-I$, then by Proposition \ref{m=1}, we know that the virtual section could only be $\tilde{\mathcal{X}}_\mathrm{flex}$ or $\tilde{\mathcal{X}}_\mathrm{sext}$. If $\Gamma_n$ does not contain $-I$, then by Lemma 6.1, $m$ is divisible by 4. and since $n=9m$, we know that 36 divides $n$.
\end{proof}

\begin{proof}[Proof of Theorem 1.4]
By Proposition \ref{classify}, any virtual section is a disjoint union of $\tilde{\mathcal{X}}^N$, $\tilde{\mathcal{X}}^F$ and $\tilde{\mathcal{X}}^S$, where 
the degree of $\tilde{\mathcal{X}}^N$ can be divided by $36$, the degree of $\tilde{\mathcal{X}}^F$ is 0 or 9, and the degree of $\tilde{\mathcal{X}}^S$ is 0 or 27. Therefore, any virtual section can only have degree 36$n$+9, 36$n$+27, or 36$n$. 
\end{proof}

\printbibliography

\end{document}